\documentclass[12pt,a4paper,reqno]{amsart}
\usepackage{graphicx}
\usepackage{amssymb}
\usepackage{bbm}
\usepackage{amsmath}
\usepackage{latexsym}
\usepackage{amscd}
\usepackage{tikz}
\usepackage{amsthm}
\usepackage{mathrsfs}
\usepackage{enumitem}
\usepackage{url}
\usepackage[english]{babel}
\usepackage{mathtools}
\usepackage[autostyle]{csquotes}
\usepackage[colorinlistoftodos]{todonotes}
\usepackage[utf8]{inputenc}
\usepackage[T1]{fontenc}
\usepackage[linktocpage=true,colorlinks,citecolor=blue,linkcolor=blue,urlcolor=blue]{hyperref}

\numberwithin{equation}{section}
\def\R{\mathbb R}
\def\Z{\mathbb Z}
\def\C{\mathbb C}

\def\N{\mathbb N}
\def\E{\mathbb E}

\def\p {\mathbbm{1}_{\wp}} 

\def\CA{\mathcal{A}}

\def\ee{\varepsilon}

\def\leq {\leqslant}

\newtheorem{theorem}{Theorem}[section]
\newtheorem{lemma}[theorem]{Lemma}
\newtheorem{proposition}[theorem]{Proposition}

\theoremstyle{remark}
\newtheorem{remark}{Remark}
\newtheorem*{xremark}{Remark}

\theoremstyle{definition}

\numberwithin{equation}{section}

\theoremstyle{remark}

\usepackage{bm,colonequals,physics}

\def\E{\mathbb E}

\title[Helson's conjecture for smooth numbers]{Helson's conjecture for smooth numbers}
\author{Seth Hardy}
\address{Mathematics Institute, Zeeman Building, University of Warwick, Coventry CV4
7AL, England}
\email{seth.hardy@warwick.ac.uk}
\author{Max Wenqiang Xu} 
\address{Courant Institute of Mathematical Sciences, 251 Mercer Street, New York 10012, USA}
\email{maxxu1729@gmail.com}
\date{\today}

\begin{document}

\begin{abstract}
Let $\Psi(x,y)$ denote the count of $y$-smooth numbers below $x$ and $P(n)$ denote the largest prime factor of $n$. We prove that for $f$ a Steinhaus random multiplicative function, the partial sums over $y$-smooth numbers always enjoy better than squareroot cancellation, in the sense that
$$ \E \Big|\sum_{\substack{1\leq n \leq x\\ P(n) \leq y}} f(n) \Big| = o\left( \Psi(x,y)^{1/2} \right),$$
uniformly on the entire range $ 2 \leq y \leq x$. The bounds are quantitative and give a large saving when $y$ isn't too close to $x$.
\end{abstract} 

\maketitle
\section{Introduction}

\subsection{Motivation and statement of results}
Multiplicative functions are a fundamental object of study in number theory, and, in recent years, research on \emph{random} multiplicative functions has been pursued as an avenue to better understand the behaviour of certain families of multiplicative functions. The study of random multiplicative functions was initiated by Wintner~\cite{Win} in 1944, where he introduced the \emph{Rademacher} random multiplicative function as a model for studying the partial sums of the Möbius function. In this paper, we concern ourselves only with the \emph{Steinhaus} random multiplicative function, $f \colon \N \rightarrow \C$, which is defined by letting $\bigl( f(p) \bigr)_{p \text{ prime}}$ be independent and identically distributed random variables uniformly distributed on the complex unit circle $\{ z \in \C : |z| = 1 \}$, and extending to composite numbers $n = p_1^{\alpha_1} \dots p_r^{\alpha_r}$ by setting 
\[
f(n) = \prod_{i=1}^r f(p_i)^{\alpha_i}.
\]
The Steinhaus random multiplicative function serves as a model for families of Dirichlet characters, $n \rightarrow \chi(n)$ for $\chi \bmod r$ chosen uniformly at random, in addition to continuous characters, $n \rightarrow n^{it}$, for $t \in [T,2T]$ chosen uniformly at random, where $r$ and $T$ are to be thought of as parameters tending to infinity. For $f$ a Steinhaus random multiplicative function, we have ``perfect orthogonality'', in the sense that
\[
\E \Bigl[ f(n) \overline{f(m)} \Bigr] = \mathbf{1}_{n=m} .
\]
This relation is mirrored by corresponding orthogonality relations for characters. It is of great interest to analytic number theorists to understand the statistical behaviour of character sums, $\sum_{n \leq x} \chi (n)$, and zeta sums, $\sum_{n \leq x} n^{it}$. In light of the orthogonality relations, the second moments of these quantities are fairly straightforward to evaluate, most of all in the Steinhaus case where we immediately see that $ \E \bigl| \sum_{n \leq x} f (n) \bigr|^2 = \lfloor x \rfloor$. A natural problem that follows is to understand the first moment, $\E \bigl| \sum_{n \leq x} f(n) \bigr|$, in addition to the character and zeta sum analogues, though understanding these quantities is a much harder problem. A naive first guess is that, ignoring the multiplicative dependence structure, the sum $\sum_{n \leq x} f(n)$ may behave like the first absolute moment of a sum of mean-zero independent random variables with unit variance, and, if this were true, the Cauchy--Schwarz upper bound $\E \bigl| \sum_{n \leq x} f(n) \bigr| \leq \sqrt{x}$ would be sharp up to a constant. This upper bound is often referred to as squareroot cancellation, and more generally gives $\E \bigl| \sum_{n \in \CA} f(n) \bigr| \leq \sqrt{|\CA|}$ for an arbitrary set $\CA$. Helson \cite{Helson} conjectured that, for the full sum up to $x$, we should actually have more than squareroot cancellation, in the sense that $\E \bigl| \sum_{n \leq x} f(n) \bigr| = o(\sqrt{x})$. The answer to this conjecture remained unclear until the problem was fully resolved by Harper \cite{HarperLow}, who showed that Helson's conjecture was true, proving the precise estimate
\[
\E \Bigl|\sum_{1\leq n \leq x} f(n) \Bigr| \asymp \frac{\sqrt{x}}{(\log \log x)^{1/4}} .  
\]
The analogous character and zeta upper bounds are also proved by Harper in \cite{harpertypical} using a remarkable comparison method. More precisely, Harper explicitly relates these moments of character sums and zeta sums to (a certain stage in the proof of) the random multiplicative case. As is to be expected, the length of the sum in that case depends on the parameters $r$ and $T$ that determine the size of the family of characters. For example, it is shown that for any large prime $r$ and for all $x\leq r$ with $x, \frac{r}{x} \to +\infty$, one has
\[
\frac{1}{r-1} \sum_{\chi \bmod r} \Big|\sum_{1\leq n \leq x}\chi(n) \Big| = o(\sqrt{x}). 
\]
In that work, Harper also conjectured (see~\cite[Equation~(1.2)]{harpertypical}) that this phenomenon can be pushed further, in that the \textit{``conductor restriction''} $x\leq r$ can be weakened if one has a M\"obius twist, an application of which would be to break the classical ``squareroot barrier'' in M\"obius cancellation (see \cite{WXRatios} for further discussion). We shall discuss later in the introduction (after Theorem~\ref{Thm: main}) about how an analogous conjecture in our setting may have consequences for counting smooth numbers in small intervals.

The key idea behind Harper's theorem is to notice that the partial sums are connected to a phenomenon known as \textit{critical Gaussian multiplicative chaos} (GMC). To briefly elaborate, we define $F_x (s) \coloneqq \prod_{p\leq x} \bigl( 1 - \frac{f(p)}{p^{s}} \bigr)^{-1}$ to be the truncated Euler product. By a non-trivial conditioning argument and an application of Parseval's identity, the proof works by first showing that
\begin{equation}\label{equ:full sum relation to EP}
\E \Bigl| \sum_{n \leq x} f(n) \Bigr| \approx \sqrt{\frac{x}{\log x}} \E \left( \int_{-1/2}^{1/2} |F_x (1/2 + it)|^{2} dt\right)^{1/2} ,
\end{equation}
and then by showing that this expectation of the Euler product integral is smaller than naively anticipated. The driving force behind this is the fact that $\bigl( \log|F(1/2 + it)| \bigr)_{t \in [-1/2,1/2]}$ is well approximated by a Gaussian field with logarithmic correlations, and the exponent 2 is exactly the critical value in this case, all together meaning that these integrals are precisely the setting where one observes critical Gaussian multiplicative chaos. The reader is encouraged to consult~\cite{Harperhigh, HarperIII} for more detailed discussions. We also remark that there have been recent progress in understanding the exact limiting distribution of the partial sums in~\cite{GW-Final, Hardy2025}.

It is natural to ask whether one still obtains ``better than squareroot cancellation'' when the full sum is replaced by a sum over a set $\CA$ that has interesting arithmetic structure. In previous work of the second author~\cite{Xu}, the case where $\CA = \CA (x, R)$ is a set of $R$-rough numbers up to $x$, i.e., all elements in the set have prime factors with size at least $R$, was studied. In that case, there is an interesting transition range, since we observe ``better than squareroot'' cancellation,
\[
\E \Bigl| \sum_{n \in \CA (x, R)} f(n) \Bigr| = o \left(\sqrt{|\CA (x, R)|} \right) ,
\]
whenever $\log \log R$ is smaller than roughly $\sqrt{\log \log x}$, but the statement fails to hold when $R$ is any larger. The case where $\mathcal{A}$ is a short interval was studied in~\cite{Caichshort}, and, together with the related work~\cite{Chang}, these three papers observe some universality in the transition threshold which is ultimately related to the ballot problem.

In this paper, our attention lies on what happens when $\CA$ is the set of $y$-smooth numbers, i.e., all elements in $\CA$ only have prime factors $p\leq y$. Here and throughout the paper, let $P(n)$ denote the largest prime factor of $n$, so that $\Psi(x,y) = \# \{ n \leq x : P(n) \leq y \}$ denotes the count of $y$-smooth numbers up to $x$. Note that the case where $y=x$ corresponds to the full sum, for which we know from Harper~\cite{HarperLow} that we have better than squareroot cancellation. Analogously to the previous cases, one might wonder if there is a certain transition range where the ``better than squareroot'' cancellation disappears. 
Our main theorem shows that this is not the case:

\begin{theorem}\label{Thm: main}
Uniformly for $2 \leq y \leq x$, we have
\[
\E \Big|\sum_{\substack{1\leq n \leq x\\ P(n) \leq y}} f(n) \Big| = o\left( \Psi(x,y)^{1/2} \right) .
\]
Quantitative bounds can be found in Theorem~\ref{Thm: main quantitative} (for moderately sized $y$), Theorem~\ref{t:small u upper bound} (for $y$ close to $x$), and Theorem~\ref{t: small y} (for small $y$).
\end{theorem}

One motivation and a potential application for our result is to make further progress on the classical problem of counting $y$-smooth numbers in short intervals $[x, x+h]$. See recent developments and discussions in \cite{soundshortinterval, KYounis}. The breakthrough work of Matomäki and Radziwi\l{}\l{} \cite{MR16} shows the existence of $x^{\ee}$-smooth numbers unconditionally in intervals of length $h\gg_{\ee} \sqrt{x}$, coming close to solving an old problem that, for any fixed $\ee > 0$, there exists an $x^{\ee}$-smooth number in $[x, x+ \sqrt{x}]$ when $x$ is sufficiently large. To see how our result may be useful, similarly to~\cite[Equation (1.2)]{harpertypical} (though removing the ``deterministic'' contribution from $t \approx 0$), we might conjecture that for any fixed $A > 0$ and small $\ee>0$, we have, very roughly speaking,
\begin{equation}\label{equ:deterministic small conductor conjecture}
\frac{1}{T} \int_{\substack{|t| \leq T \\ t \not \approx 0}} \Bigl| \sum_{\substack{n \leq x \\ P(n) \leq x^{\ee}}} n^{-it} \Bigr| \, dt \ll \E \Bigl| \sum_{\substack{n \leq x \\ P(n) \leq x^{\ee}}} f(n) \Bigr|, \quad \forall x \leq T^A .
\end{equation}
Importantly, we know from Theorem~\ref{Thm: main} that the right-hand side is $o \bigl(\sqrt{\Psi(x,x^{\ee})} \bigr)$. Similar to as described in~\cite{harpertypical}, though isolating the main term in this case, Perron's formula suggests that
\[
\frac{\Psi (x + h, x^{\ee}) - \Psi (x,x^{\ee})}{h} \approx \frac{1}{2\pi x} \int_{-x/h}^{x/h} \sum_{\substack{n \leq x + h \\ P(n) \leq x^{\ee}}} n^{-it} \biggl( \frac{x}{h} \biggr)^{it} \, dt \approx \frac{\Psi(x,x^{\ee})}{x} + E(x,x^{\ee},h),
\]
where
\[
E (x,x^{\ee},h) \ll \frac{1}{x} \int_{\substack{|t| \leq x/h \\ t \not \approx 0}} \Bigl| \sum_{\substack{n \leq x + h \\ P(n) \leq x^{\ee}}} n^{-it} \Bigr| \, dt .
\]
In particular, since $\Psi(x,x^\varepsilon) \asymp_\varepsilon x$, combining Theorem~\ref{Thm: main} and~\eqref{equ:deterministic small conductor conjecture} gives the bound $E (x,x^{\varepsilon},\sqrt{x}) = o\big(  \sqrt{\Psi(x, x^{\ee})/x} \big) = o (1)$, which (very roughly) gives an asymptotic for the count in intervals of length $h = \sqrt{x}$. In fact, our quantitative result for $y$ in this range, Theorem~\ref{t:small u upper bound}, would suggest further that this strategy could show the existence of $x^{\ee}$-smooth numbers in intervals of length $h \gg \sqrt{x} / (\log \log x)^{1/4}$, similarly to how the analogous conjecture in~\cite{harpertypical} gives M\"{o}bius cancellation in sums of length $h \gg \sqrt{x} / (\log \log x)^{1/4}$.

\bigskip 
In this paper, we observe three distinct sources for the ``better than squareroot cancellation'' phenomena in Theorem~\ref{Thm: main}; the size of $y$ being key to this analysis. First of all, for $y$ close to $x$, one should expect that the critical GMC phenomenon uncovered by Harper~\cite{HarperLow} should give better than squareroot cancellation. We find this to be the case, and moreover, we are able to use critical GMC phenomenon to prove better than squareroot cancellation even when $y$ is as small as $e^{(\log \log x)^{5/3+\varepsilon}}$. We outline how this is done in Section~\ref{s: large y explained}, and a precise statement is given in Theorem~\ref{t:small u upper bound}. The critical GMC saving in our case specifically comes from exploiting the randomness of $f(p)$ on the small primes (to be precise, primes smaller than roughly $e^{\frac{1}{1 - \alpha}}$ for $\alpha = \alpha(x,y)$ the saddle point, see Section~\ref{s:smooths intro} for details), and, as one may expect, the saving is always quite small (like a power of $\log \log x$, or smaller).

Given the previous works~\cite{Caichshort, HarperLow, Xu}, one may not expect to observe cancellation that is larger than some power of $\log \log x$ for any reasonably large $y$. \emph{However, we find that when $y$ is smaller than roughly $x^{\frac{1}{\log \log x}}$, we can obtain an unexpectedly large saving (coming from the primes closer to $y$).} In short, this saving comes from the fact that the dominant contribution to our Euler product expectations comes from highly unlikely events (we explain this in detail in Section~\ref{s:outline of moderate y}). The flavour of this might remind some readers of the \textit{supercritical} GMC case, although this is not exactly what we observe, since our Euler products are shifted off the half-line and are no longer approximated by the exponential of a log-correlated Gaussian field. This allows us to prove the following:

\begin{theorem}[Moderate $y$]\label{Thm: main quantitative}
Let $u = \frac{\log x}{\log y}$ and fix $\varepsilon > 0$. For $(\log x)^{1 + \varepsilon} \leq y \leq x^{\frac{1}{(\log \log x)^{1 + \varepsilon}}}$, we have
\begin{equation}\label{equ: alt statement}
\E \Bigl| \sum_{\substack{n \leq x \\ P(n) \leq y}} f(n) \Bigr| \ll \sqrt{\Psi(x,y) \exp \bigl( -u (\log 2 + o(1) \bigr)},
\end{equation}
where the $o(1)$ term goes to zero uniformly for $y$ in this range. This result is derived from the following bound: In the larger range $1000 \log x \leq y \leq x^{\frac{1}{\log \log x}} $, we have
\[
\E \Bigl| \sum_{\substack{n \leq x \\ P(n) \leq y}} f(n) \Bigr| \ll \Psi(x^2, y)^{1/4} \exp \biggl( O \biggl( u \Bigl( \frac{\log x}{y} \Bigr)^{1/3} \biggr) \biggr) (\log x)^{9/8} (\log y)^{1/8},
\]
where all implied constants are uniform for $y$ in this range. As is shown in the proof (see~\eqref{equ: comp to trivial bound}), this
gives better than squareroot cancellation (saving a factor of at least $e^{cu}$ over the trivial bound $\sqrt{\Psi(x,y)}$, for some fixed $c>0$) uniformly on the entire range $C \log x \leq y \leq x^{\frac{1}{3 \log \log x}} $ for some large constant $C \geq 1000$.
\end{theorem}
This theorem is proved in Section~\ref{sec: moderate}, and we stress that the proof \emph{does not} make use of critical GMC phenomena, where the saving would typically be much smaller. This is our main result, and we outline the proof in Section~\ref{s:outline of moderate y}. 

Finally, when $y$ is small (below roughly $\log x$), we again have a natural transition range in the behaviour of smooth numbers, and the methods used to prove the above theorem fail. Despite this, a new analysis of random Euler products close to the imaginary axis allows us to obtain lots of cancellation (see Theorem~\ref{t: small y}). 

Before outlining the proof of Theorem~\ref{Thm: main quantitative} and stating our results for other ranges of $y$, we first recall some fundamental results on smooth numbers.

\subsection{Smooth numbers and the transition range}\label{s:smooths intro}
For the uninitiated reader, we give a short review of smooth numbers, restricting ourselves to the case where $(\log x)^3 \leq y \leq x$, say. We refer readers to \cite{Gran08} for more background. 

We can quite accurately upper bound the number of $y$-smooth numbers below $x$ simply by a simple application of Rankin's trick. Define the truncated Euler product $\zeta(s,y) \coloneqq \prod_{p \leq y} \bigl( 1 - \frac{1}{p^s} \bigr)^{-1}$ for any $\Re (s) > 0$. For any $\sigma > 0$, we have
\begin{equation}\label{equ:rankins trick}
\Psi(x,y) = \sum_{\substack{n \leq x \\ P(n) \leq y}} 1 \leq \sum_{P(n) \leq y} \biggl( \frac{x}{n} \biggr)^{\sigma} = x^\sigma \prod_{p \leq y} \biggl( 1 - \frac{1}{p^{\sigma}} \biggr)^{-1} = x^{\sigma} \zeta (\sigma, y) .
\end{equation}
We now choose $\sigma > 0$ so that it minimises the right-hand side. It turns out that this minimiser is unique, and is called the saddle point, $\alpha = \alpha (x,y)$. Since $\sigma = \alpha(x,y)$ satisfies $\frac{d}{d \sigma} \log (x^{\sigma} \zeta(\sigma, y)) \mid_{\sigma = \alpha (x,y)} = 0$, it follows that
\[
\sum_{p \leq y} \frac{\log p}{p^{\alpha (x,y)}-1} = \log x .
\]
and it can be shown~\cite[Lemmas~1 and~2]{HildTen86} that
\begin{equation}\label{equ: Saddle point approx in intro}
\alpha (x,y) = 1 - \frac{\log (u \log (u + 1))}{\log y} + O \biggl( \frac{1}{\log y} \biggr),
\end{equation}
where $u = \frac{\log x}{\log y}$. The error term can be improved to $o\bigl( \frac{1}{\log y} \bigr)$ as soon as $(\log x)^2 \leq y \leq x^{o(1)}$, and we state a precisely result in Lemma~\ref{l:saddle point approx}. The bound obtained by Rankin's trick, $\Psi(x,y) \leq x^{\alpha} \zeta(\alpha,y)$, is close to optimal, seeing as for all $2 \leq y \leq x$, we have
\begin{equation}\label{equ:rankin to true count comparison}
x^{\alpha} \zeta(\alpha,y) \ll \Psi(x,y) \sqrt{\log x \log y},
\end{equation}
which follows from~\cite[Theorems~1 and~2]{HildTen86}. In our analysis, in contrast to~\eqref{equ:full sum relation to EP}, we are interested in integrals involving the partial Euler products $F_y (\alpha / 2 + it)$. In this case, the variance of $\log |F_y (\alpha/2 + it)|$ is roughly $\frac{1}{2} \sum_{p \leq y} \frac{1}{p^{\alpha}}$, and 
letting $z = e^{\frac{1}{1 - \alpha}}$, this can be written as
\begin{equation}\label{equ: variance evaluation}
\begin{split}
\frac{1}{2} \sum_{p \leq y} \frac{1}{p^{\alpha}}
& = \frac{1}{2} \biggl( \sum_{p \leq z} \frac{1}{p^{1 - \frac{1}{\log z}}} + \sum_{z < p \leq y} \frac{1}{p^{\alpha}} \biggr) \\
& \sim \frac{1}{2} \biggl( \log \Bigl( \frac{1}{1 - \alpha} \Bigr) + \frac{y^{1 - \alpha}}{(1 - \alpha) \log y} \biggr),
\end{split}
\end{equation}
whenever $2/3 \leq \alpha \leq 1 - \frac{1}{\log y}$, say. Calculations of this type can be found in~\cite[Lemma~7.4]{MV2007}. Using the saddle point approximation~\eqref{equ: Saddle point approx in intro}, the first term in the parentheses is roughly $\log \log y$, and the second term is of size $u$. Therefore, the dominant contribution to this variance comes from the primes $p \leq z$ whenever $u$ is smaller than roughly $\log \log x$, whereas the dominant contribution comes from $p > z$ whenever $u$ is larger. This is why we see a phase transition at roughly $y \approx x^{\frac{1}{\log \log x}}$, though we will elaborate more on this in the next section.

\subsection{Outline of the proof of Theorem~\ref{Thm: main quantitative}}\label{s:outline of moderate y} 
In this section, we will explain the behaviour for ``moderately sized'' $y$ (though this actually encompasses a large range), say $C \log x \leq y \leq x^{\frac{1}{\log \log x}}$ for some large constant $C>0$. The typical way that one approaches these expectation problems is to use conditioning to relate the first absolute moment of our partial sums to the $1/2$'th moment of the mean square of a random Euler product, as in~\eqref{equ:full sum relation to EP}. In our case, this turns out to be unnecessary, and we can begin by applying Perron's formula to obtain a bound of the form
\begin{equation}\label{equ: Perron application}
\E \Bigl| \sum_{\substack{n \leq x \\ P(n) \leq y}} f(n) \Bigr| \ll \E \biggl| \int_{\sigma/2 - ix}^{\sigma/2 + ix} F_y (s) x^s \frac{ds}{s} \biggr|,
\end{equation}
where $\sigma > 0$ will be chosen later, and we remind the reader that the random Euler product is defined as $F_y(s) \coloneqq \prod_{p\leq y} (1-\frac{f(p)}{p^{s}})^{-1}$. One may suspect that the correct choice of $\sigma$ that minimises the right-hand side would be $\sigma (x,y) = \alpha(x,y)$, where $\alpha (x,y)$ is the saddle point, however will see that this is actually sub-optimal. By an application of the triangle inequality, we have
\begin{equation}\label{equ:triangle}
\E \Bigl| \sum_{\substack{n \leq x \\ P(n) \leq y}} f(n) \Bigr| \ll x^{\sigma/2} \int_{-x}^{x} \E \biggl| \frac{F_y (\sigma / 2 + it)}{\sigma / 2 + it} \biggr| \, dt . 
\end{equation}
On first glance, this application of the triangle inequality looks like it should be wasteful, since we are losing any possible cancellation in the oscillatory integral. We will later explain why this is not the case on most of our range of $y$. 

Proceeding, note that for any fixed $t \in \R$, the distribution of $\bigl( f(p) \bigr)_{p \leq y}$ is identical to the distribution of $( f(p) p^{-it} )_{p \leq y}$, so it follows that $\E |F_y (\sigma / 2 + it)| = \E | F_y (\sigma/2)|$. Therefore, we obtain the bound
\[
\E \Bigl| \sum_{\substack{n \leq x \\ P(n) \leq y}} f(n) \Bigr| \ll x^{\sigma/2} \E |F_y (\sigma/2)| (\log x) ,
\]
so long as $\sigma$ is not too small. Now, for any fixed $\sigma > 2/3$, say, the quantity $\log |F_y (\sigma / 2)|$ is approximately Gaussian with zero mean and variance $\sim \frac{1}{2} \sum_{p \leq y} \frac{1}{p^{\sigma}}$. The moment generating function is therefore
\begin{equation}\label{equ: gamma}
 \E |F(\sigma / 2 )|^\gamma = \E \exp \Bigl( \gamma \log | F_y (\sigma / 2 ) | \Bigr) \approx \exp \Bigl( \frac{\gamma^2}{4} \sum_{p \leq y} \frac{1}{p^{\sigma}} \Bigr) \approx \zeta (\sigma, y)^{\frac{\gamma^2}{4}}.   
\end{equation}
Note that the quadratic growth in the exponent means that $\E |F_y (\sigma/2)|$ can be significantly smaller than $\sqrt{\E |F_y (\sigma/2)|^2}$, and the latter bound corresponds to the square root of the Rankin's trick upper bound~\eqref{equ:rankins trick}. We prove such a result in Lemma~\ref{l:exp result moderate y}, which gives the above on order of magnitude whenever $\sigma > 2/3$, before the higher prime powers begin taking effect. Similar calculations for $\sigma$ close to one can be found in~\cite[Euler Product Result 1]{Harperhigh}. Taking $\gamma = 1$, and supposing that~\eqref{equ: gamma} holds as an upper bound, we have
\[
\E \Bigl| \sum_{\substack{n \leq x \\ P(n) \leq y}} f(n) \Bigr| \ll x^{\sigma/2} \zeta(\sigma, y)^{1/4} (\log x) . 
\]
The first two terms on the right-hand side can be written as $\bigl( (x^2)^{\sigma} \zeta(\sigma,y) \bigr)^{1/4}$, which is minimised by choosing $\sigma$ to be the saddle point $\sigma = \alpha (x^2,y)$. Taking this value of $\sigma$, and applying~\eqref{equ:rankin to true count comparison}, we have
\[
\E \Bigl| \sum_{\substack{n \leq x \\ P(n) \leq y}} f(n) \Bigr| \ll \Psi(x^2, y)^{1/4} (\log x)^{9/8} (\log y)^{1/8}.
\]
which gives our theorem in the range $C \log x \leq y \leq x^{\frac{1}{\log \log x}}$. We now explain why this argument should be relatively efficient, certainly in the range $(\log x)^{1 + \varepsilon} \leq y \leq x^{\frac{1}{(\log \log x)^{1 + \varepsilon}}}$. In the remainder of this section, for compactness, we set $\alpha = \alpha (x^2,y)$. Also, for simplification, we restrict our attention to the study of the bounded integral,
\[
\E \biggl| \int_{-1/2}^{1/2} F_y (\alpha/2 + it) x^{it} \, dt \biggr|.
\]
This suffices for our understanding since breaking up~\eqref{equ: Perron application} will only increase our bound by a logarithmic factor. Now, similarly, to as mentioned before, each fixed $t \in [-1/2, 1/2]$, the quantity $\log| F_y (\alpha/2 + it)|$ is approximately Gaussian with zero mean and variance $\approx \frac{1}{2} \sum_{p \leq y} \frac{1}{p^{\alpha}}$. These Gaussians should begin to decorrelate when $t$ varies by $\approx \frac{1}{\log y}$, so letting $G_j$ be independent random variables with distribution $\mathcal{N}(0, \frac{1}{2} \sum_{p \leq y} \frac{1}{p^{\alpha}})$, one might suspect that
\begin{equation}\label{equ: max}
\max_{t \in [-1/2, 1/2]} \log |F_y (\alpha / 2 + it)| \approx \max_{|j| \leq \log y/2} G_j \approx \biggl( \log \log y \sum_{p \leq y} \frac{1}{p^{\alpha}} \biggr)^{1/2} ,
\end{equation}
with high probability. In contrast, assuming Gaussian behavior of $\log |F_y (\alpha/2 + it)|$, one can find that the dominant contribution to $ \E |F(\alpha / 2 + it )|^\gamma$ comes from increasingly large values of $\log |F_y (\alpha / 2 + it)|$, specifically from events where $\log |F_y (\alpha / 2 + it)| = \frac{\gamma}{2} \sum_{p \leq y} \frac{1}{p^{\alpha}} + O \bigl(\bigl( \sum_{p \leq y} \frac{1}{p^{\alpha}} \bigr)^{1/2} \bigr)$. Similar ideas are present in the analysis of moments of the Riemann zeta function, for example, in the works of Soundararajan~\cite{SoundMoment} and Harper~\cite{harper2013sharp}. Therefore, by comparing to~\eqref{equ: max}, if $\alpha$ is sufficiently smaller than $1$, then the dominant contribution to $\E |F(\alpha/2 + it)|$ comes from events where $\log |F_y (\alpha / 2 + it)|$ is significantly larger than the probable size of the maximum on $[-1/2, 1/2]$, that is, \textit{the moments $\E |F(\alpha/2 + it)|$ are controlled by very unlikely events}, and the calculation in~\eqref{equ: variance evaluation} highlights that this transition happens roughly when $y \approx x^{\frac{1}{\log \log x}}$. To see how this motivates the efficiency of our argument, let $E(t)$ denote the event that $\log |F_y (\alpha/2 + it)| > (\frac{1}{2} - \varepsilon) \sum_{p \leq y} \frac{1}{p^{\alpha}}$, and let $E^c (t)$ be its complement. We have
\[
\E \biggl| \int_{-1/2}^{1/2} F_y (\alpha/2 + it) x^{it} \, dt \biggr| = \E \biggl| \int_{-1/2}^{1/2} \bigl( \mathbf{1}_{E(t)} + \mathbf{1}_{E^c (t)} \bigr) F_y (\alpha/2 + it) x^{it} \, dt \biggr|.
\]
Seeing as $E(t)$ includes the dominant contribution to the first moment, one can calculate that $\E \mathbf{1}_{E^c (t)} |F_y (\alpha / 2 + it)|$ is relatively small, and so the right-hand side is approximately
\[
\E \biggl| \int_{-1/2}^{1/2} \mathbf{1}_{E(t)} F_y (\alpha/2 + it) x^{it} \, dt \biggr| .
\]
Now, if $E(t)$ occurs for any $t \in [-1/2, 1/2]$, then it must be the case that
\[
\max_{t \in [-1/2,1/2]} \log |F_y (\alpha / 2 + it)| > \Bigl(\frac{1}{2} - \varepsilon \Bigr) \sum_{p \leq y} \frac{1}{p^{\alpha}}. 
\] 
Comparing this to~\eqref{equ: max}, we see that the above expectation is dominated by events where the maximum of $\log |F_y (\alpha/2 + it)|$ is unusually large. When this happens, we would only expect to find one interval of length $\ll \frac{1}{\log y}$ where values of the size of this maximum occur, and the Euler product would typically be relatively small at all other points. This means that the dominant contribution to the above integral comes from a single interval where the Euler product is roughly fixed, suggesting that the use of the triangle inequality in~\eqref{equ:triangle} should be efficient. This analysis is closely related to the study of high moments of sums of random multiplicative functions~\cite{Harperhigh}, and we note that the above argument suggests a possible route for proving a similar lower bound. Furthermore, for $0<q<1$, our result suggests that we may have
\[
\E \Bigl| \sum_{\substack{n \leq x \\ P(n) \leq y}} f(n) \Bigr|^{2q} \ll_q \Psi (x^{1/q}, y)^{q^2} ,
\]
up to some lower-order factors. We do not pursue these lines of enquiry here.

\subsection{Large smoothness parameter: Multiplicative Chaos}\label{s: large y explained}
To cover the remaining ranges and prove Theorem~\ref{Thm: main}, we first describe how one handles the somewhat delicate range when $y$ is very close to $x$. Our analysis actually allows us to prove better than squareroot cancellation on a larger range of $y$ (though Theorem~\ref{Thm: main quantitative} is always superior whenever $y \leq x^{\frac{1}{4 \log \log x}}$). Specifically, we are able to prove the following:
\begin{theorem}[Large $y$]\label{t:small u upper bound}
For any $e^{(\log \log x)^2} \leq y \leq x$, say, we have
\[
\E \biggl| \sum_{\substack{n \leq x \\ P(n) \leq y}} f(n) \biggr| \ll \frac{\Psi (x,y)^{1/2}}{\Bigl(\log \min \Bigl\{ \frac{1}{1-\alpha(x,y)}, \log x \Bigr\} \Bigr)^{1/4}},
\]
where the implied constant is uniform in $y$.
\end{theorem}

\begin{xremark}
By an application of Lemma~\ref{l:saddle point approx}, in the range $x^{\frac{1}{\log \log x}} \leq y \leq x$, this result gives a saving of size $(\log \log x)^{1/4}$, similarly to~\cite{HarperLow}.
\end{xremark}

To outline the main ideas behind this result, we begin by assuming that $y$ is close to $x$, where we might expect to be close to the critical GMC regime studied in~\cite{HarperLow}. In that work, having established an upper bound of the form~\eqref{equ:full sum relation to EP}, the key step involves showing that 
\begin{equation}\label{equ: critical chaos ep from harper}
\E \biggl( \int_{-1/2}^{1/2} |F_x (1/2 + it)|^2 \, dt \biggr)^{1/2} \ll \biggl( \frac{\log x}{\sqrt{\log \log x}} \biggr)^{1/2},
\end{equation}
giving better than squareroot cancellation. Analogously to the previous section, we need to prove a statement along the lines of
\[
\E \biggl( \int_{-1/2}^{1/2} |F_y (\alpha/2 + it)|^2 \, dt \biggr)^{1/2} = o \left( \sqrt{\zeta(\alpha,y)} \right) ,
\]
where $\alpha = \alpha(x,y)$. As mentioned, the main difference between our case and the work of Harper~\cite{HarperLow} is that our Euler products \emph{are not} evaluated on the half-line. This means that the random process $\bigl( \log |F_y (\alpha/2 + it)| \bigr)_{t \in [-1/2,1/2]}$ does not necessarily behave like a log-correlated Gaussian field, and so it is not clear that we can connect our Euler products to critical Gaussian multiplicative chaos. The key observation is that the portion of the Euler product coming from $p \leq e^{\frac{2}{1 - \alpha}}$, say, \emph{does} behave like it is evaluated on the half-line. To be precise, taking $z = e^{\frac{2}{1 - \alpha}}$, we have
\[
F_z (\alpha / 2 + it) = F_{z} \biggl( 1/2  - \frac{1}{\log z} + it \biggr).
\]
Note that the specific choice of $2$ in the numerator of $\frac{2}{1 - \alpha}$ is unimportant\footnote{In Section~\ref{Sec: GMC}, we choose a slightly different constant to allow for direct use of a result from~\cite{HarperLow}. This is specifically seen in Lemma~\ref{l:Euler product on half line}.} and is just chosen so that the second term is exactly $\frac{1}{\log z}$. The important point (which is shown rigorously in~\cite{HarperLow}) is that a shift of order $\frac{1}{\log z}$ has no impact on the behaviour of these Euler products, so similarly to~\eqref{equ: critical chaos ep from harper}, we have
\[
\E \biggl( \int_{-1/2}^{1/2} \Bigl| F_z \Bigl( \frac{1}{2} + \frac{1}{\log z} + it \Bigr) \Bigr|^2 \, dt \biggr)^{1/2} \ll \biggl( \frac{\log z}{\sqrt{\log \log z}} \biggr)^{1/2} \asymp \Biggl( \frac{1}{(1 - \alpha) \sqrt{\log \frac{1}{1 - \alpha}}} \Biggr)^{1/2} .
\]
This is the contents of Lemma~\ref{l:Euler product on half line}, and it allows us to obtain better than squareroot cancellation on these partial Euler products over small primes. The key idea then is that we only ``win'' on products over small primes. By the triangle inequality, it suffices to handle sums roughly of the form 
\[
\E \Bigl| \sum_{\substack{n \leq x \\ z < p \leq y}} f(n) \Bigr| = \E \Bigl| \sum_{\substack{1 < m \leq x \\ p | m \Rightarrow p \in (z, y]}} f(m) \sum_{\substack{n \leq x/m \\ P(n) \leq z}} f(n) \Bigr|,
\]
We can apply the Cauchy--Schwarz inequality to the conditional expectation at this stage in the proof, leaving only the randomness from these small primes $p \leq z$. In comparison to~\cite{HarperLow}, the new input needed is a precise short intervals bound for numbers with restricted prime factorisation. This is provided by Lemma~\ref{l: smooths in short intervals estimate}.

\subsection{Sums over very smooth numbers}\label{s: small y explained}
Given Theorems~\ref{Thm: main quantitative} and~\ref{t:small u upper bound}, to complete the proof of Theorem~\ref{Thm: main}, we need to handle the case $2 \leq y \leq C \log x$ (for a large constant $C$). We begin by mentioning the most extreme case: when $y=2$, we have
\[
\E \Bigl| \sum_{\substack{n \leq x \\ P(n) \leq 2}} f(n) \Bigr| = \E \Bigl| \sum_{n \leq \lfloor \frac{\log x}{\log 2} \rfloor} f(2)^n \Bigr| = \int_0^1 \Bigl| \sum_{n \leq \lfloor \frac{\log x}{\log 2} \rfloor} \mathrm{e} (n \theta) \Bigr| \, d \theta \ll \log \log x,
\]
which is very small relative to $\Psi(x,2)^{1/2} \asymp (\log x)^{1/2}$. A similar analysis becomes difficult when, for example, $y \geq \log \log x$, so instead our analysis largely follows the same blueprint as described for moderately sized $y$, in Section~\ref{s:outline of moderate y} (applying Perron's formula followed by an estimate for the first absolute moment of the Euler product). However, when $y \leq C \log x$, say, one can no longer expect a result such as~\eqref{equ: gamma} to hold, seeing as we need to account for primes with higher multiplicity (see Remark~\ref{r:EP transition range}). Despite this, we are still able to show cancellation in the Euler products (Lemma~\ref{l:exp for small y}), which allows us to prove the following:
\begin{theorem}[Small $y$]\label{t: small y} 
The following results hold:
\begin{enumerate}[label=(\roman*)]
\item There exists some constant $D \geq 5$ such that, for $D \leq y \leq \log x/\log \log x$, say, we have
\[
\E \Bigl| \sum_{\substack{n \leq x \\ P(n) \leq y}} f(n) \Bigr| \ll  x^{\alpha(x,y)/2} (\log x) \prod_{p \leq y} \biggl( \frac{1}{3} \log \biggl( \frac{1}{1-p^{-\alpha(x,y)/2}} \biggr) \biggr) + \log x . 
\]
In particular, the previous display is
\[
\ll  \sqrt{\Psi(x,y)} (\log x)^{5/4} (\log y)^{1/4} \biggl( \sqrt{\frac{y}{\log x}} \log \biggl( \frac{\log x}{y} \biggr) \biggr)^{\pi(y) - \pi(\sqrt{y})} + \log x . 
\]
where $\pi (y) = \# \{ p \leq y \}$. This gives better than squareroot cancellation in the range $D \leq y \leq \log x/\log \log x$.
\item For any large fixed constant $C \geq 1$, in the range $\log x / \log \log x \leq y \leq C \log x$, we have
\[
\E \Bigl| \sum_{\substack{n \leq x \\ P(n) \leq y}} f(n) \Bigr| \ll \sqrt{\Psi (x,y)/\exp ((\log x)^{c})}, 
\]
where $c > 0$ is a small constant depending on $C$.
\item Suppose that $2 \leq y \leq D$ for any large constant $D$ (in particular, the one from (i)), then we have
\[
\E \Bigl| \sum_{\substack{n \leq x \\ P(n) \leq y}} f(n) \Bigr| \ll \sqrt{\Psi(x,y)\frac{\log \log x}{\log x}}
\]
where the implied constant depends on $D$.
\end{enumerate}

\end{theorem}
This result shows that we have better than squareroot cancellation uniformly for $y \leq C \log x$ (for any large fixed constant $C>0$). We have not made a significant attempt to optimise these results, instead choosing to provide a fairly concise proof of the qualitative result, Theorem~\ref{Thm: main}, in this range. That being said, we believe that the calculations performed in this proof, in particular Lemma~\ref{l:exp for small y}, may be of particular interest to some readers.

\subsection{Organization} 
In Section~\ref{Sec: smooth}, we introduce some (mostly classical) results about smooth numbers. Importantly, we give a useful estimate for short sums over integers with restricted prime factorisations, Lemma~\ref{l: smooths in short intervals estimate}. In Section~\ref{Sec: Euler}, we prove results for random Euler products with the feature that the evaluation is not necessarily on the half-line. Section~\ref{sec: moderate} contains the proof of Theorem~\ref{Thm: main quantitative}, which handles a large range of smoothness parameters, followed by Section~\ref{Sec: GMC}, which allows us to handle the case when $y$ is close to $x$. Finally, in Section~\ref{Sec: small y}, we prove Theorem~\ref{t: small y} to cover the ``very smooth'' case. Appendix~\ref{s:appendix} contains a proof of Lemma~\ref{l: smooths in short intervals estimate}, which follows similarly to previous work of Hildebrand~\cite{Hildebrand}.

\subsection*{Acknowledgement}
The authors would like to thank Adam Harper, Carl Schildkraut and K. Soundararajan for their interest in this paper and for helpful discussions. SH is supported by the Swinnerton-Dyer scholarship at the Warwick Mathematics Institute Centre for Doctoral Training.
MWX is supported by a Simons Junior Fellowship from the Simons Foundation. Part of the work was done during several visits of MWX at Warwick Mathematics Institute, and the warm hospitality is greatly appreciated. 

\section{Smooth Number Results}\label{Sec: smooth}

In this section, we state some important results about smooth numbers that are used frequently throughout the paper. First of all, we state some estimates for the saddle point and its relation to the count of smooth numbers. We then state an estimate for sums over integers that have restrictions on their prime factorization, which follows similarly to previous work of Hildebrand~\cite{HildTen86}, and whose proof is delayed until Appendix~\ref{s:appendix}.

We begin by stating some classical results on smooth numbers are frequently used throughout the paper. 

\begin{lemma}[Saddle point estimate 1]\label{l:saddle point approx}  
For $2 \log x \leq y \leq x$, say, $u= \frac{\log x}{\log y}$ and for $\alpha(x, y)$ the saddle point, we have
\[
\alpha (x,y) = 1 - \frac{\log (u \log (u + 1))}{\log y} + O \biggl( \frac{1}{\log y} \biggr).
\] 
In fact, for $2 \log x \leq y \leq x^{1/10}$, say, we have 
\[
\alpha (x,y) = 1 - \frac{\log (u \log u)}{\log y} + O \biggl( \frac{\log \log u}{(\log y) (\log u)} + \frac{\log x}{y \log y}\biggr).
\] 
\end{lemma}
\begin{proof}
These estimates follow from~\cite[Lemmas~1 and~2]{HildTen86}.
\end{proof}
In certain cases, such as when $y$ is quite small, it will be convenient to use the following formulation:
\begin{lemma}[Saddle point estimate 2]\label{l:saddle point approx uniform}  
Uniformly for $2 \leq y \leq x$, we have
\[
\alpha (x,y) = \frac{\log ( 1 + y/\log x )}{\log y} \biggl( 1 + O \biggl( \frac{\log \log (1+y)}{\log y} \biggr) \biggr).
\] 
\end{lemma}
\begin{proof}
This is~\cite[Theorem~2]{HildTen86}.
\end{proof}

\begin{lemma}[Explicit smooth count]\label{l:smooth count in terms of saddle point}
Uniformly for $x \geq y \geq 2$, the number of $y$ smooth numbers up to $x$ satisfies
\[
\Psi(x, y)=\frac{x^\alpha \zeta(\alpha, y)}{\alpha \sqrt{2 \pi(1+(\log x) / y) \log x \log y}}\left(1+O\left(\frac{1}{\log (u+1)}+\frac{1}{\log y}\right)\right),
\]
where $\alpha = \alpha (x,y)$ is the saddle point.
\end{lemma}
\begin{proof}
This is~\cite[Theorem~1]{HildTen86}.
\end{proof}

\begin{lemma}[Smooth number comparison 1]\label{l: smooths below x/d vs below x}
For $2 \leq y \leq x$, and $1 \leq d \leq x$, we have
\[
\Psi(x/d, y) \ll \frac{1}{d^{\alpha}} \Psi(x, y) ,
\]
where $\alpha = \alpha (x, y)$ denotes the saddle-point corresponding to the $y$-smooth numbers less than $x$.
\end{lemma}
\begin{proof}
This is~\cite[Théorème~2.4(i)]{BretecheTenenbaum}.
\end{proof}

In Section~\ref{Sec: GMC}, it will be necessary to estimate short sums over integers whose prime factors all lie in a given range. For this task, we employ the following lemma:

\begin{lemma}[Prime restricted sum]\label{l: smooths in short intervals estimate}
Fix any small $\varepsilon > 0$ and let $x \geq 2$ be a parameter tending to infinity. Suppose $y, h, \delta$ are parameters such that $e^{(\log \log x)^{5/3 + \varepsilon}} \leq y \leq x^{1/\delta}$, $x/y^{1/3} \leq h \leq x/2$, and $\frac{1}{\log y} \leq \delta \leq \frac{1}{10}$. Uniformly in these parameters, we have
\[
\frac{1}{h} \sum_{\substack{x < n \leq x + h \\ p \mid n \Rightarrow p \in (y^\delta, y]}} 1 \ll \frac{1}{x} \cdot \frac{\Psi(x,y)}{\delta \log y} .
\]
\end{lemma}
As mentioned, the proof of this will follow similarly to Hildebrand~\cite[Theorem~3]{Hildebrand}, and we prove it in Appendix~\ref{s:appendix}. Here the lower bound on $y$ comes from the current best known error term in the prime number theorem. In the case of long intervals, this sum is asymptotically evaluated in~\cite[Théorème~2.1]{BretecheTenenbaum} and~\cite[Theorem~1]{Xuan}. Our result should incur no loss when $\delta \ll \frac{1}{\log u}$, seeing as then the smooth and rough conditions behave roughly as though they are independent of one another.  One can also find more general sums of this form in~\cite{Friedlander}.

\section{Euler product results}\label{Sec: Euler}

In this section, we collect useful results on partial Euler products. The main results of interest, Lemmas~\ref{l:exp result moderate y},~\ref{l:exp for small y}, and ~\ref{l:Euler product on half line}, give estimates for random Euler products which will later be key to the proofs of Theorems~\ref{Thm: main quantitative},~\ref{t: small y} and~\ref{t:small u upper bound}, respectively. The main feature of these results (different from most other work in the area) is that they study the partial Euler products off the line $\Re(s) = 1/2$. Despite this, we will find that our results always beat the ``trivial bound'' in certain cases. It may be helpful for the reader to observe that, for any $\sigma > 0$, we have
\begin{equation}\label{equ: trivial EP bound}
\E |F_y (\sigma / 2 + it)|^2 = \E \biggl| \sum_{\substack{P(n) \leq y}} \frac{f(n)}{n^{\sigma/2 + it}} \biggr|^2 = \zeta (\sigma, y),
\end{equation}
by expanding out the square and applying orthogonality. Note that we can exchange the order of summation and expectation due to the fact that our sums are absolutely bounded. Therefore, if we use Cauchy--Schwarz to bound the first absolute moment in equation~\eqref{equ:triangle}, this delivers a bound analogous to Rankin's trick (with an additional logarithmic factor from the Perron integral). We now prove a result that will allow us to directly evaluate the first moment of our Euler product integrals which we will use in the proof of Theorem~\ref{Thm: main quantitative}. In particular, so long as $\sigma$ isn't too small, this result will give a saving over the trivial bound for the first moment, since it shows that $\E |F_y (\sigma/2+it)|$ is significantly smaller than $\sqrt{\E|F_y (\sigma / 2 + it)|^2} = \zeta(\sigma, y)^{1/2}$.

\begin{lemma}[Main expectation result]\label{l:exp result moderate y}
Let $f$ be a Steinhaus random multiplicative function, let $\beta$ be a real number such that $|\beta| \leq 2$, say. Suppose also that $t \in \R$, and $0 < \sigma \leq 2$, then we have
\[
\E |F_y (\sigma / 2 + it)|^{\beta} = |\zeta (\sigma, y)|^{\beta^2/4} \exp \biggl( O \biggl( \sum_{p \leq y} \frac{1}{p^{3 \sigma / 2} | 1 - p^{-\sigma/2}|} \biggr) \biggr).
\]
where the implied constant is absolute.
\end{lemma}
\begin{remark}\label{r:EP transition range}
The error term in this result becomes very large when $\sigma$ is very small, in particular when $\sigma \ll \frac{1}{\log y}$, corresponding to inefficiency in the Taylor expansion of the logarithm. When one takes $\sigma = \alpha(x,y)$, this corresponds to an inefficiency in the range $y \ll \log x$. This is a manifestation of the change in the anatomy of smooth numbers, where, for $y \ll \log x$ a typical integer has prime factors that occur with high multiplicity.
\end{remark}
\begin{proof}
The result proceeds similarly to previous results for similar quantities, such as~\cite[Euler Product Result 1]{Harperlargevalue}. We first note that, for any $t \in \R$, we have
\[
\E |F_y (\sigma / 2 + it) |^{\beta} = \E |F_y (\sigma / 2) |^{\beta},
\]
which follows from translation invariance in law. Expanding out the definition of $F_y (\sigma / 2)$ and using the fact that the individual Euler factors are independent, we have
\[
\E |F_y (\sigma / 2) |^{\beta} = \prod_{p \leq y} \E \biggl| 1 - \frac{f(p)}{p^{\sigma/2}} \biggr|^{-\beta} ,
\]
In the ball $|z| < 1$, we have the truncated Taylor expansion $\log (1 - z) = - z - \frac{z^2}{2} + O \bigl( \frac{|z|^3}{1 - |z|} \bigr)$ for the principal branch of the logarithm. We stress that the implied constant is uniform for $|z| < 1$. Therefore the right-hand side of the previous display can be written as
\begin{multline*}
\exp \biggl( - \beta \Re \log \Bigl( 1 - f(p) / p^{\sigma /2} \Bigr) \biggr) = \\ 
\exp \biggl( \beta \Re \bigl( f(p) \bigr) / p^{\sigma / 2} + \beta \Re \bigl( f(p)^2 \bigr) / p^{\sigma} + O \Bigl( 1 /\bigl( p^{3 \sigma / 2} |1 - p^{- \sigma/2}| \bigr) \Bigr) \biggr).
\end{multline*}
by Taylor expansion of the logarithm. We now Taylor expand the first two terms in the exponential, noting that each of these terms is uniformly bounded. We deduce that
\begin{multline*}
\E \biggl| 1 - \frac{f(p)}{p^{\sigma/2}} \biggr|^{-\beta} = 
\E \biggl( 1 + \beta \Re \bigl( f(p) \bigr) / p^{\sigma / 2} + \\ \beta^2 \Re \bigl( f(p) \bigr)^2 / 2 p^{\sigma} + \beta \Re \bigl( f(p)^2 \bigr) / p^{\sigma} \bigr) \biggr) \exp \biggl( O \biggl( \frac{1}{p^{3 \sigma / 2} |1 - p^{- \sigma / 2} |} \biggr) \biggr) .
\end{multline*}
It follows by symmetry that $\E \Re \bigl( f(p) \bigr) = \E \Re \bigl( f(p)^2 \bigr) = 0$. Furthermore, we have $\E \Re \bigl( f(p) \bigr)^2 = \frac{1}{4} \E \bigl( f (p) + \overline{f(p)} \bigr)^2 = 1/2 $, so that the previous display is equal to 
\[
\biggl( 1 + \frac{\beta^2}{4 p^{\sigma}} \biggr) \exp \biggl( O \biggl( \frac{1}{p^{3 \sigma / 2} | 1 - p^{-\sigma/2}|} \biggr) \biggr) = \exp \biggl( \frac{\beta^2}{4p^{\sigma}} + O \biggl( \frac{1}{p^{3 \sigma / 2} | 1 - p^{-\sigma/2}|} \biggr).
\]
Importantly, the implicit constants in the ``big Oh'' term do not need to depend on $\beta$, $\sigma$, or $p$. We deduce that
\[
\E |F_y (\sigma / 2 + it) |^{\beta} = \exp \biggl( \frac{\beta^2}{4} \sum_{p \leq y} \frac{1}{p^{\sigma}} + O \biggl( \sum_{p \leq y} \frac{1}{p^{3 \sigma / 2} | 1 - p^{-\sigma/2}|} \biggr) \biggr) .
\]
Since it follows from Taylor expansion (similarly to the above) that $\log |\zeta (\sigma, y)| = \sum_{p \leq y} 1/p^{\sigma} + O \bigl( \sum_{p \leq y}\frac{1}{p^{2 \sigma} |1 - p^{-\sigma}|} \bigr)$, and noting that this error term is dominated by the one above, the statement of the lemma follows. 
\end{proof}

As mentioned in Remark~\ref{r:EP transition range}, the previous result becomes inefficient when $\sigma \ll \frac{1}{\log y}$, meaning that we cannot use it to analyse our sum over $y$-smooth numbers for $y \leq C \log x$ (where $C>0$ is a large constant). Instead, in this range of small $y$ (specifically, in the proof of Theorem~\ref{t: small y}), we will employ the following result for the first moment:

\begin{lemma}[Expectation result close to imaginary axis]\label{l:exp for small y}
Let $f$ be a Steinhaus random multiplicative function. There exists a fixed $\ee>0$ such that the following holds. Suppose that $t \in \R$ and $0 < \sigma \leq \frac{\varepsilon}{\log y}$. Then we have
\[
\E |F_y (\sigma / 2 + it)| \leq \prod_{p \leq y} \biggl( \frac{1}{3}\log\Big(\frac{1}{1-p^{-\sigma/2}}\Big) \biggr).
\]
\end{lemma}
\begin{proof}
Again, we make use of translation invariance to note that $\E_y |F_y (\sigma + it)| = \E |F_y (\sigma/2)|$. Then using independence and the fact that each $f(p)$ is uniformly distributed on the complex unit circle, we have
\[
\E |F_y (\sigma / 2)| = \prod_{p \leq y} \E \bigl| 1 - f(p) p^{- \sigma/2} \bigr|^{-1} = \prod_{p \leq y} \int_{0}^{1} \frac{d \theta}{| 1 - p^{- \sigma/2} e^{2 \pi i \theta}|}.
\]
We will show that there exists a constant $0< \eta< 1$ such that for $\eta < c < 1$, we have 
\begin{equation}\label{eqn: I}
   I: = 
\int_{0}^{1} \frac{d \theta}{| 1 - c e^{2 \pi i \theta}|} \leq \frac{1}{3} \log \biggl( \frac{1}{1-c} \biggr) .
\end{equation} Note that the integral $I$ can be rewritten as, after change of variable $m = \frac{4c}{(1+c)^{2}}$,
$I  = \frac{2}{\pi (1+c) } \cdot K(m)$, 
where $K(m) = \int_{0}^{\pi/2} \frac{d\theta}{\sqrt{1-m\sin^{2}\theta}}$ is the so called complete elliptic integral of the first kind. In particular, it is known and easy to see that as $m \to 1^{-}$ it has asymptotic $K(m)\sim \log (\frac{4}{\sqrt{1-m}})$ (see, e.g. \cite{Elliptic}).
 It follows that as $c \to 1^{-}$,  
\[I = \frac{1}{\pi } \log (\frac{1}{1-c}) + \frac{3\log 2}{\pi} + o(1).  \]
Thus, by choosing $\eta< c<1$ close enough to 1, one can see \eqref{eqn: I} holds.   
By setting $e^{-\ee/2} := \eta$, we apply \eqref{eqn: I} with $c= p^{-\sigma/2} = e^{-\sigma \log p /2}$ where $\sigma = \ee/\log y$ and this concludes the proof. 
\end{proof}

We now state the final Euler product lemma in this section, which makes use of critical GMC phenomena. This will allow us to obtain cancellation in random Euler product integrals on the line $\Re (s) = \alpha/2$ provided that the truncation length is small enough in terms of $\alpha$, which forms the main ingredient in the proof of Theorem~\ref{t:small u upper bound}.
\begin{lemma}[Expectation result close to half line]\label{l:Euler product on half line}
Let $f$ be a Steinhaus random multiplicative function and let $c = 2e^{-2}$. For $4/5 \leq \alpha < 1$, we define $z = e^\frac{c}{1 - \alpha}$, so that $F_{z} (s) = \prod_{p \leq e^{\frac{c}{1 - \alpha}}} \bigl( 1 - \frac{f(p)}{p^{s}} \bigr)^{-1}$. Then we have
\[
\E \biggl( \int_{-1/2}^{1/2} \Bigl| F_{z} \Bigl( \frac{\alpha}{2} + it \Bigr) \Bigr|^2 \, dt \biggr)^{2/3} \ll \Biggl( \frac{1}{(1 - \alpha) \sqrt{\log \frac{1}{1 - \alpha}}} \Biggr)^{2/3} .
\]
\end{lemma}
\begin{proof}
Let $x \geq 10$, and define $F_{x, k} (s) \coloneqq \prod_{p \leq x^{e^{-(k+1)}}} \bigl( 1 - \frac{f(p)}{p^s} \bigr)^{-1}$. We then have
\[
\E \Biggl( \frac{e^{k} (1-q) \sqrt{\log \log x}}{\log x} \int_{-1/2}^{1/2} \biggl| F_{x, k} \biggl( 1/2  - \frac{k}{\log x} + it \biggr) \biggr|^2 \, dt \Biggr)^q \ll 1 ,
\]
uniformly for all $0 \leq k \leq \lfloor \log \log x \rfloor$ and $2/3 \leq q \leq 1 - \frac{1}{\sqrt{\log \log x}}$. This is proved by Harper~\cite{HarperLow} in the section titled: ``Proof of the upper bound in Theorem 1, assuming Key Propositions 1 and 2''. To obtain the lemma, we apply this result with $k=1$, $x=e^{\frac{1}{1 - \alpha}}$, and $q = 2/3$. Since
\[
    \E \biggl( \int_{-1/2}^{1/2} \Bigl| F_{z} \Bigl( \frac{\alpha}{2} + it \Bigr) \Bigr|^2 \, dt \biggr)^{2/3} = \E \biggl( \int_{-1/2}^{1/2} \Bigl| F_{x, 1} \Bigl( 1/2 - \frac{1}{\log x} + it \Bigr) \Bigr|^2 \, dt \biggr)^{2/3},
\]
the result follows.
\end{proof}

Finally, we introduce a version of Parseval's identity, now ubiquitous in the study of random multiplicative functions. This will be used in conjunction with the above lemma to prove Theorem~\ref{t:small u upper bound}.
\begin{lemma}[Multiplicative Parseval identity]\label{l:ha result}
Let $(a_n)_{n=1}^\infty$ be a sequence of complex numbers, and let $A(s) = \sum_{n=1}^\infty \frac{a_n}{n^s}$ denote the corresponding Dirichlet series, and $\sigma_c$ the abscissa of convergence. Then for any $\sigma > \max\{ 0, \sigma_c \}$, we have \[ \int_0^\infty \frac{|\sum_{1\leq n \leq x} a_n|^2}{x^{1 + 2 \sigma}} \, dx \, = \frac{1}{2\pi} \int_{-\infty}^{\infty} \biggl| \frac{A(\sigma + i t)}{\sigma + it} \biggr|^2 \, dt \, . \]
\end{lemma} 
\begin{proof}
This is \cite[Equation~(5.26)]{MV2007}.
\end{proof}

\section{Moderately smooth numbers: Proof of Theorem~\ref{Thm: main quantitative}}\label{sec: moderate}

In this section, we perform the rigorous analysis outlined in Section~\ref{s:outline of moderate y}, the main input being Lemma~\ref{l:exp result moderate y}, in addition to classical results for smooth numbers found in Section~\ref{Sec: smooth}.

\begin{proof}[Proof of Theorem~\ref{Thm: main quantitative}]

To begin, suppose that $1000 \log x \leq y \leq x$. By Perron's formula~\cite[Corollary~5.3]{MV2007}, we have
\[
\sum_{\substack{n \leq x \\ P(n) \leq y}} f(n) = \frac{1}{2 \pi i} \int_{1 - i x}^{1 + ix} F_y (s) x^s \frac{ds}{s} + O \Biggl( 1 + \sum_{\substack{x/2 < n \leq x \\ n \neq x}} \min \biggl\{ 1, \frac{1}{|x-n|} \biggr\} + \prod_{p \leq y} \biggl( 1 - \frac{1}{p} \biggr)^{-1} \Biggr).
\]
It follows from straightforward calculations (using Mertens' Theorem to bound the Euler product) that the error term here is $\ll \log x$. 

Since the integrand is analytic and has no poles apart from $s=0$, we can shift the integral to the line $\Re(s) = \alpha(x^2, y)/2$. Specifically, we have
\begin{multline*}
\frac{1}{2 \pi i} \int_{1 - i x}^{1 + ix} F_y (s) x^s \frac{ds}{s} = \frac{1}{2 \pi i} \int_{\alpha (x^2, y) - ix}^{\alpha (x^2, y) + ix} F_y (s) x^s \frac{ds}{s} \\  
+ O \biggl( \frac{1}{x} \int_{\alpha (x^2, y)/2}^1 \bigl( |F_y (\sigma / 2 + ix)| + |F_y (\sigma / 2 - ix)| \bigr) x^{\sigma} \, d \sigma \biggr) .
\end{multline*}
For the remainder of this section, we let $\alpha \coloneqq \alpha (x^2, y)$. Not carefully that, in the following sections, we will instead take $\alpha$ to be the typical saddle point corresponding to $x$ and $y$. By the triangle inequality, it follows from the previous equations that 
\begin{equation}\label{eqn: whole}
   \begin{split}
       \E \Bigl| \sum_{\substack{n \leq x \\ P(n) \leq y}} f(n) \Bigr| \ll  x^{\alpha(x^2, y)/2} \int_{-x}^{x} \E \biggl| \frac{F_y (\alpha / 2 + it)}{\alpha / 2 + it} \biggr| \, dt + \log x \\ 
+ \frac{1}{x} \int_{\alpha/2}^1 \bigl( \E |F_y (\sigma / 2 + ix)| + \E |F_y (\sigma / 2 - ix)| \bigr) x^{\sigma} \, d \sigma. 
   \end{split} 
\end{equation}
We now apply Lemma~\ref{l:exp result moderate y} to each expectation here. For the last term, we note that the bound is maximised on our range of integration when $\sigma = \alpha/2$, so that we have $\E |F_y (\sigma / 2 + it)| \ll \zeta(\alpha, y)^{1/4} \exp \bigl( O \bigl( \sum_{p \leq y} \frac{1}{p^{3 \alpha / 2} | 1 - p^{-\alpha /2}|} \bigr) \bigr)$. Then using the fact that $\int_{-x}^{x} \frac{dt}{|\alpha/2 + it|} \ll \log x$ uniformly on our range of $\alpha = \alpha (x^2, y)$ (see Lemma~\ref{l:saddle point approx}), we find that the first term dominates the last term, and putting this all together, we deduce that
\begin{equation}\label{equ: moderate y before triangle}
\E \Bigl| \sum_{\substack{n \leq x \\ P(n) \leq y}} f(n) \Bigr| \ll x^{\alpha (x^2, y)/2} \zeta \bigl( \alpha (x^2, y), y \bigr)^{1/4} \exp \biggl( O \biggl( \sum_{p \leq y} \frac{1}{p^{3 \alpha / 2} |1 - p^{-\alpha / 2}|} \biggr) \biggr) (\log x).
\end{equation}
Here we have also used the fact that our bound certainly exceeds $\log x$, so the first term dominates the right-hand side of \eqref{eqn: whole}.

We proceed to bound the prime number sum in the exponent. For $\alpha \log p \leq 1$, we have the lower bound $|1 - p^{-\alpha / 2}| \gg \alpha \log p$, and for $\alpha \log p \geq 1$, we have $|1 - p^{-\alpha / 2}| \gg 1$. When $x$ is sufficiently large, we certainly have $\alpha (x^2, y) \log y \geq 1$ for $y \geq 1000 \log x$ by Lemma~\ref{l:saddle point approx uniform}, so we deduce that
\begin{equation}\label{eqn: primesum}
\sum_{p \leq y} \frac{1}{p^{3 \alpha / 2} |1 - p^{-\alpha / 2}|} \ll \sum_{p \leq e^{1/\alpha}} \frac{1}{\alpha p^{3 \alpha / 2} \log p} + \sum_{e^{1 / \alpha} < p \leq y} \frac{1}{p^{3 \alpha / 2}} . 
\end{equation}
Suppose now that $y \geq (\log x)^{4}$, then we certainly have $\alpha (x^2, y) \geq 7/10$, say, hence $3 \alpha (x^2, y) / 2 \geq 1.05$. It then follows from the above that
\[
\sum_{p \leq y} \frac{1}{p^{3 \alpha / 2} |1 - p^{-\alpha / 2}|} \ll 1
\]
Now suppose that $1000 \log x \leq y \leq (\log x)^4$. Then we know by Lemma~\ref{l:saddle point approx uniform} that $\frac{1}{\log y} \leq \alpha (x^2, y) \leq 4/5$ when $x$ is sufficiently large, and it follows from dyadic decomposition that
\[
\sum_{p \leq e^{1/\alpha}} \frac{1}{\alpha p^{3 \alpha / 2} \log p} \ll 1 + \frac{1}{\alpha} \sum_{1 \leq i \leq 1/\alpha} \frac{e^{i(1 - 3 \alpha / 2)}}{i^2} \ll e^{1/\alpha}  ,
\]
say. In this range of $y$, i.e., $1000 \log x \leq y \leq (\log x)^4$, we also have
\[
\sum_{e^{1 / \alpha} < p \leq y} \frac{1}{p^{3 \alpha / 2}} \ll 1 + \frac{y^{1 - 3 \alpha / 2}}{\log y} \]
 Hence in view of \eqref{eqn: primesum} and combining the above estimates, we deduce that
\[
\sum_{p \leq y} \frac{1}{p^{3 \alpha / 2} |1 - p^{-\alpha / 2}|} \ll e^{1/\alpha} + \frac{y^{1 - 3 \alpha / 2}}{\log y} ,
\]
uniformly for all $1000 \log x \leq y \leq x$. Inserting this estimate into~\eqref{equ: moderate y before triangle}, we have
\[
\E \Bigl| \sum_{\substack{n \leq x \\ P(n) \leq y}} f(n) \Bigr| \ll x^{\alpha (x^2, y) / 2} \zeta \bigl( \alpha (x^2, y), y \bigr)^{1/4} \exp \biggl( O \biggl( e^{1/\alpha} + \frac{y^{1 - 3 \alpha / 2}}{\log y} \biggr) \biggr) (\log x) .
\]
Now applying equation~\eqref{equ:rankin to true count comparison} (which follows from Lemmas~\ref{l:smooth count in terms of saddle point} and~\ref{l:saddle point approx}), we have
\begin{equation}\label{equ: quant bound in terms of Psi (x^2, y)}
\E \Bigl| \sum_{\substack{n \leq x \\ P(n) \leq y}} f(n) \Bigr| \ll \Psi(x^2, y)^{1/4} \exp \biggl( O \biggl( e^{1/\alpha} + \frac{y^{1 - 3 \alpha / 2}}{\log y} \biggr) \biggr) (\log x)^{9/8} (\log y)^{1/8}.
\end{equation}
We now bound the error terms in the exponential more explicitly. In particular, we will show that
\[
e^{1/\alpha} + \frac{y^{1 - 3 \alpha / 2}}{\log y} \ll 1 + u \biggl( \frac{\log x}{y} \biggr)^{1/3}. 
\]
For this, we will make use of Lemma~\ref{l:saddle point approx uniform}, which (for ease of reference) tells us that
\begin{equation}\label{equ: saddle point uniform}
\alpha (x,y) = \frac{\log (1 + y/\log x)}{\log y} \biggl( 1 + O \biggl( \frac{\log \log (1+y)}{\log y} \biggr) \biggr)
\end{equation}
uniformly for $2 \leq y \leq x$. Note that we wish to apply this with $\alpha = \alpha(x^2,y)$, and in particular it implies that $\alpha (x^2, y) \geq \frac{8 \log (1 + y / 2 \log x)}{9 \log y}$, say, so for $1000 \log x \leq y \leq x$, we have
\[
y^{-3\alpha / 2} \ll \exp \biggl( - \frac{4}{3} \log \biggl(1 + \frac{y}{2\log x}\biggr) \biggr) \ll \biggl( \frac{\log x}{y} \biggr)^{4/3}.
\]
From this it follows that
\[
\frac{y^{1 - 3 \alpha / 2}}{\log y} \ll u \biggl( \frac{\log x}{y} \biggr)^{1/3}.
\]
Furthermore, by~\eqref{equ: saddle point uniform}, we have
\[
e^{1/\alpha} \ll \exp \biggl( \frac{2 \log y}{\log (1 + y/2\log x)} \biggr).
\]
In the range $y \geq (\log x)^2$, say, this is certainly $\ll 1$, and for $1000 \log x \leq y \leq (\log x)^2$, this satisfies the bound $\ll y^{3/10}$. These estimates together leads to that the bound
\[
e^{1/\alpha} \ll 1 +  u \biggl( \frac{\log x}{y} \biggr)^{1/3},
\]
holds uniformly for $1000 \log x \leq y \leq x$. Now inserting these estimates into~\eqref{equ: quant bound in terms of Psi (x^2, y)}, we have 
\begin{equation}\label{equ: L1 norm in terms of psi(x^2, y) bound}
\E \Bigl| \sum_{\substack{n \leq x \\ P(n) \leq y}} f(n) \Bigr| \ll \Psi(x^2, y)^{1/4} \exp \biggl( O \biggl( u \Bigl( \frac{\log x}{y} \Bigr)^{1/3} \biggr) \biggr) (\log x)^{9/8} (\log y)^{1/8},
\end{equation}
This proves the second formulation in Theorem~\ref{Thm: main quantitative} on the range $1000 \log x \leq y \leq x$, though as we will see, the result is only non-trivial in a more restrictive range of $y$, so it is not stated for this range in the theorem. To elucidate our saving, we now undo the application of Lemma~\ref{l:smooth count in terms of saddle point} so that we can compare the right-hand side to the trivial bound, $\sqrt{\Psi(x,y)}$. In particular, we will show that in the range $1000 \log x \leq y \leq x^{\frac{1}{\log \log x}}$, we have
\[
\Psi (x^2, y) \ll \frac{\Psi(x,y)^2}{\log y} \exp \biggl( - u \biggl( 2 \log 2 + O \biggl( \Bigl( \frac{\log x}{y} \Bigr)^{1/3} + \frac{\log \log u}{\log u} \biggr) \biggr) \biggr),
\]
and from this we will deduce~\eqref{equ: alt statement}. First of all, by Lemma~\ref{l:smooth count in terms of saddle point}, we have
\begin{equation}\label{equ: order of mag comparrison}
\Psi (x^2, y) \asymp \Psi (x,y)^2 x^{2 (\alpha (x^2, y) - \alpha (x,y))} \frac{\zeta \bigl( \alpha (x^2, y), y \bigr)}{\zeta \bigl( \alpha (x,y), y \bigr)^2} \frac{\alpha(x,y)^2}{\alpha(x^2, y)} \sqrt{\log x \log y}.
\end{equation}
We begin by considering the $x^{2 (\alpha (x^2, y) - \alpha (x,y))}$ term. Applying the second part of Lemma~\ref{l:saddle point approx}, we have
\begin{equation}\label{equ:alpha(x^2,y) vs alpha(x,y)}
\begin{split}
\alpha (x^2, y) - \alpha (x,y) &= \frac{\log (\frac{u \log u}{2 u \log 2u})}{\log y} + O \biggl( \frac{\log \log u}{(\log y)(\log u)} \biggr) \\
& = - \frac{\log 2}{\log y} + O \biggl( \frac{\log \log u}{(\log y)(\log u)} \biggr) .
\end{split}
\end{equation}
on the range $1000 \log x \leq y \leq x^{\frac{100}{\log \log x}}$, say. This is precisely where our saving of the form $\exp (- u \log 2 / 2)$ comes from. We now move on to bounding the quotient of zeta functions. First of all, we note that for any $\sigma > 0$, we have
\[
\log \zeta (\sigma, y) = \sum_{p \leq y} \frac{1}{p^{\sigma}} + O \biggl( \sum_{p \leq y} \frac{1}{p^{2 \sigma} |1 - p^{-\sigma}|} \biggr).
\]
This follows from the fact that $\log (1 - x) = - x + O \bigl( \frac{|x|^2}{1 - |x|} \bigr) $ uniformly for $|x| < 1$. For both $\sigma = \alpha (x,y)$ and $\sigma = \alpha (x^2, y)$, the handling of the error term is identical to the similar term that appeared previously, specifically the error term in the exponential in equation~\eqref{equ: moderate y before triangle}. By following identically the handling of that term (or by noting that the terms in our case are slightly smaller), we can deduce the same bound here, so that 
\[
\frac{\zeta \bigl( \alpha (x^2, y), y \bigr)}{\zeta \bigl( \alpha (x,y), y \bigr)^2} = \exp \biggl( \sum_{p \leq y} \frac{1}{p^{\alpha(x^2, y)}} - 2 \sum_{p \leq y} \frac{1}{p^{\alpha (x,y)}} + O \biggl(1 +  u \Bigl( \frac{\log x}{y} \Bigr)^{1/3} \biggr) \biggr).
\]
We now bound the prime number sums, which we evaluate using~\cite[Lemma~7.4]{MV2007} to find that
\[
\sum_{p \leq y} p^{-\alpha (x, y)} = \frac{y^{1 - \alpha(x,y)}}{(1 - \alpha(x,y)) \log y} +  \log \Bigl( \frac{1}{1 - \alpha(x,y)} \Bigr) + O \biggl( \frac{y^{1 - \alpha(x,y)}}{(1 - \alpha(x,y))^2 (\log y)^2} \biggr).
\]
Applying Lemma~\ref{l:saddle point approx}, we obtain
\[
\sum_{p \leq y} p^{-\alpha (x, y)} = u + \log \Bigl( \frac{1}{1 - \alpha(x,y)} \Bigr) + O \biggl( \frac{u \log \log u}{\log u} + \frac{u \log x}{y}\biggr),
\]
uniformly for $2 \log x \leq y \leq x^{\frac{100}{\log \log x}}$, say. Similarly, in the same range, we have 
\[
\sum_{p \leq y} p^{-\alpha (x^2 , y)} = 2u + 2 \log \Bigl( \frac{1}{1 - \alpha(x^2 ,y)} \Bigr) + O \biggl( \frac{u \log \log u}{\log u} + \frac{u \log x}{y}\biggr).
\]
Therefore, we obtain the bound
\[
\frac{\zeta \bigl( \alpha (x^2, y), y \bigr)}{\zeta \bigl( \alpha (x,y), y \bigr)^2} = \Biggl( \frac{\bigl( 1 - \alpha (x,y) \bigr)^2}{1 - \alpha (x^2 ,y)} \Biggr)  \exp \biggl( O \biggl( u \Bigl( \frac{\log x}{y} \Bigr)^{1/3} + \frac{u \log \log u}{\log u} \biggr) \biggr).
\]
Combining this result with~\eqref{equ: order of mag comparrison} and~\eqref{equ:alpha(x^2,y) vs alpha(x,y)}, we deduce that
\begin{multline*}
\Psi (x^2, y) = \Psi(x,y)^2 \exp \biggl( - u \biggl( 2 \log 2 + O \biggl( \Bigl( \frac{\log x}{y} \Bigr)^{1/3} + \frac{\log \log u}{\log u} \biggr) \biggr) \biggr) \\ \times \biggl( \frac{\alpha(x,y)^2}{\alpha(x^2, y)} \biggr) \Biggl( \frac{\bigl( 1 - \alpha (x,y) \bigr)^2}{1 - \alpha (x^2 ,y)} \Biggr) \sqrt{\log x \log y}.
\end{multline*}
Finally, we note that on the range $2 \log x \leq y \leq x^{\frac{100}{\log \log x}}$, say, Lemma~\ref{l:saddle point approx} implies that $\frac{\alpha(x,y)^2}{\alpha(x^2, y)} \ll \frac{1}{\log y}$, and by Lemma~\ref{l:saddle point approx} we obtain the bound $\bigl( \frac{( 1 - \alpha (x,y) )^2}{1 - \alpha (x^2 ,y)} \bigr) \sqrt{\log x \log y} \ll \sqrt{u} \log u$, which can be engulfed into the error in the exponential. Therefore, as desired, we have
\[
\Psi (x^2, y) \ll \frac{\Psi(x,y)^2}{\log y} \exp \biggl( - u \biggl( 2 \log 2 + O \biggl( \Bigl( \frac{\log x}{y} \Bigr)^{1/3} + \frac{\log \log u}{\log u} \biggr) \biggr) \biggr),
\]
which holds uniformly for $1000 \log x \leq y \leq x^{\frac{1}{\log \log x}}$, say. Inserting this estimate into equation~\eqref{equ: L1 norm in terms of psi(x^2, y) bound} and using the fact that $(\log x)^{9/4} (\log y)^{-1/4} \ll u \log x$, we have
\begin{equation}\label{equ: comp to trivial bound}
\E \Bigl| \sum_{\substack{n \leq x \\ P(n) \leq y}} f(n) \Bigr| \ll \\ 
\sqrt{\Psi(x,y)\exp \Bigl(-u \biggl( \log 2 -  O \biggl(\Bigl( \frac{\log x}{y} \Bigr)^{1/3} + \frac{\log \log u}{\log u} \biggr) \biggr) \biggr) (\log x)^2}.
\end{equation}
uniformly for $1000 \log x \leq y \leq x^{\frac{1}{\log \log x}}$. Now, in the range $(\log x)^{1 + \varepsilon} \leq y \leq x^{\frac{1}{(\log \log x)^{1 + \varepsilon}}}$, say, we have $\frac{\log x}{y} = o(1)$ and $\log \log x = o(u)$, where the $o(1)$ terms goes to zero uniformly for fixed $\varepsilon > 0$. Therefore, equation~\eqref{equ: alt statement} holds on this range of $y$, as required. It is also clear from this inequality that we have better than squareroot cancellation (for smaller $y$) whenever $y \geq C \log x$ where $C \geq 1000$ is a large constant, and for $y$ closer to $x$ we have better than squareroot cancellation whenever $u > \frac{2}{\log 2} \log \log x \approx 2.885 \log \log x$ to overcome the $(\log x)^{2}$ term. Therefore, for some large constant $C \geq 1000$, the range $C \log x \leq y \leq x^{\frac{1}{3\log \log x}}$ is such that we always have a saving of at least $e^{cu}$ for some fixed $c > 0$ on this range. This completes the proof of Theorem~\ref{Thm: main quantitative}.
\end{proof}

\section{Large smoothness parameter: Proof of Theorem~\ref{t:small u upper bound}}\label{Sec: GMC}

In this section, we prove Theorem~\ref{t:small u upper bound}. As mentioned in Section~\ref{s: large y explained}, the proof essentially follows that of Harper~\cite{HarperLow}, but in our case, we only condition on relatively small primes. Since the remaining random Euler product behaves as though it is on the half-line, we are in a similar setting to Harper~\cite{HarperLow}, and we can apply Lemma~\ref{l:Euler product on half line} to obtain a non-trivial saving. The additional input needed for this strategy is an estimate for the count of smooth numbers in short intervals without small prime factors, which is provided by Lemma~\ref{l: smooths in short intervals estimate}.

The theorem will follow by iteratively applying the following proposition:
\begin{proposition}\label{p:small u upper bound}
Let $v>2$ be a large constant so that $\alpha (x,y) \leq 1 - \frac{10}{\log y}$ whenever $y \leq x^{1/v}$ (it can be verified using Lemma~\ref{l:saddle point approx} that such a constant exists), and let $c = 2 e^{-2}$.  For any $e^{(\log \log x)^{2}} \leq y \leq x^{1/v}$, we have
\[
\E \biggl| \sum_{\substack{n \leq x \\ e^{\frac{c}{1-\alpha}} < P(n) \leq y}} f(n) \biggr| \ll \frac{\Psi (x,y)^{1/2}}{\Bigl(\log \frac{1}{1-\alpha(x,y)} \Bigr)^{1/4}},
\]
where the implied constant is uniform in $y$. 
\end{proposition}
\begin{xremark}
We reiterate that the constant $c = 2 e^{-2}$ is unimportant, and is just chosen to allow for straightforward application of Lemma~\ref{l:Euler product on half line} later on.
\end{xremark}

\begin{proof}[Proof of Theorem~\ref{t:small u upper bound}, assuming Proposition~\ref{p:small u upper bound}]
For compactness, in this section, we will write $\alpha$ in place of $\alpha(x,y)$. We begin by noting that, in the case where $u = \frac{\log x}{\log y} \leq v$ (i.e. when $x^{1/v} \leq y \leq x$), Theorem~\ref{t:small u upper bound} follows identically\footnote{For example, one can follow the proof of Proposition 1 in~\cite[Section~2.4]{HarperLow}, just removing a finite number of $k$ in the initial splitting of the sum. Since this will only decrease the overall size, the same upper bound follows.} to the work of Harper~\cite{HarperLow}. Now suppose that $y \leq x^{1/v}$. By the triangle inequality, we have
\begin{equation}\label{equ:ub triangle ineq}
\E \biggl| \sum_{\substack{n \leq x \\ P(n) \leq y}} f(n) \biggr| \leq \E \biggl| \sum_{\substack{n \leq x \\ e^{\frac{c}{1-\alpha}} < P(n) \leq y}} f(n) \biggr| + \E \biggl| \sum_{\substack{n \leq x \\ P(n) \leq e^{\frac{c}{1-\alpha}}}} f(n) \biggr| .
\end{equation}
If $u \geq \log \log \log x$, then the first term can easily be shown to dominate. Specifically, in this range of $u$, we apply Proposition~\ref{p:small u upper bound} to handle the first term and Cauchy--Schwarz inequality to handle the second term (together with a mean square calculation), giving
\[
\E \biggl| \sum_{\substack{n \leq x \\ P(n) \leq y}} f(n) \biggr| \leq \frac{\Psi(x,y)^{1/2}}{\Bigl(\log \frac{1}{1-\alpha(x,y)} \Bigr)^{1/4}} + \Psi(x,e^{\frac{c}{1 - \alpha}} )^{1/2} .
\]
It follows readily from Lemma~\ref{l:saddle point approx} that $e^{\frac{c}{1 - \alpha}} \leq y^{1/2}$, say, and a short calculation using Lemma~\ref{l:smooth count in terms of saddle point} allows one to deduce that
\[
\Psi(x,\sqrt{y}) \ll \Psi(x,y) e^{-u \log u/3} \ll \frac{\Psi(x,y)}{\Bigl(\log \frac{1}{1-\alpha(x,y)} \Bigr)^{1/2}} ,
\]
and the result follows immediately. 

It remains to handle the case where $v \leq u \leq \log \log \log x$. The main idea is to iteratively apply the decomposition ~\eqref{equ:ub triangle ineq} and then use Proposition~\ref{p:small u upper bound} to evaluate the resulting absolute first moment. Define $y_0 = y$, and $y_{i+1} = e^{\frac{c}{1 - \alpha (x, y_i)}}$ for $i \geq 0$. Iterating~\eqref{equ:ub triangle ineq} $\ell$-times, we have
\[
\E \biggl| \sum_{\substack{n \leq x \\ P(n) \leq y}} f(n) \biggr| \leq \sum_{k=0}^{\ell-1} \E \biggl| \sum_{\substack{n \leq x \\ y_{k+1} < P(n) \leq y_k}} f(n) \biggr| + \E \biggl| \sum_{\substack{n \leq x \\ P(n) \leq y_{\ell}}} f(n) \biggr| .
\]
We apply Proposition~\ref{p:small u upper bound} to bound the terms in the first sum and the Cauchy--Schwarz inequality to the latter term, giving
\[
\E \biggl| \sum_{\substack{n \leq x \\ P(n) \leq y}} f(n) \biggr| \leq \sum_{k=0}^{\ell-1} \frac{\Psi(x,y_k)^{1/2}}{\Bigl(\log \frac{1}{1-\alpha( x,y_k )} \Bigr)^{1/4}} + \Psi(x, y_{\ell}) .
\]
It follows from our conditions on $y$ and Lemma~\ref{l:saddle point approx} that
$y_1 = e^{\frac{c}{1 - \alpha (x,y_0)}} \leq \sqrt{y_0} = \sqrt{y}$, and similarly we have
$y_k \leq y^{1/2^k}$. Taking $\ell = \lfloor \log \log \log x \rfloor$, say, we have
\[
\E \biggl| \sum_{\substack{n \leq x \\ P(n) \leq y}} f(n) \biggr| \leq \frac{1}{(\log \log x)^{1/4}} \sum_{k=0}^{\ell-1} \Psi \bigl( x,y^{1/2^k} \bigr)^{1/2} + \Psi(x, y^{1/2^{\ell}}) .
\]
We now make use of the relation between $\Psi(x,y)$ and the Dickman function, $\rho$. It follows from~\cite[Theorem~III.5.8]{TenenbaumAPNT} that $\Psi \bigl( x,y^{1/2^k} \bigr) \asymp x \rho ( 2^k u )$ uniformly for $k$ in our range of summation. Applying the estimate $\frac{\rho(2u)}{\rho(u)} \leq \frac{1}{2}$ that holds for all $u \geq 1$ (since $\rho$ is rapidly decreasing, this ratio is maximised at $u=1$), we deduce that $\Psi \bigl( x,y^{1/2^k} \bigr) \ll \frac{x}{2^k} \rho(u) \ll \frac{x}{2^k} \Psi(x,y)$. Finally, since $\Psi(x, y^{1/2^{\ell}}) \ll \Psi(x, x^{\frac{1}{\log \log x}}) \ll \Psi(x,y) / \log \log x$, say, we deduce that
\[
\E \biggl| \sum_{\substack{n \leq x \\ P(n) \leq y}} f(n) \biggr| \leq \frac{\Psi(x,y)^{1/2}}{(\log \log x)^{1/4}} ,
\]
which agrees with Theorem~\ref{t:small u upper bound} on this range of $u$. This completes the proof.
\end{proof}
We now proceed with the proof of Proposition~\ref{p:small u upper bound}.
\begin{proof}[Proof of Proposition~\ref{p:small u upper bound}]
Recall that $c = 2e^{-2}$. We have
\[
\E \biggl| \sum_{\substack{n \leq x \\ e^{\frac{c}{1-\alpha}} < P(n) \leq y}} f(n) \biggr| = \E \biggl| \sum_{\substack{1 < m \leq x \\ p | m \Rightarrow p \in (e^{\frac{c}{1-\alpha}}, y]}} f(m) \sum_{\substack{n \leq x/m \\ P(n) \leq e^{\frac{c}{1 - \alpha}}}} f(n) \biggr|.
\]
Let ${\E}_{\alpha}$ denote the expectation conditioned on primes $p \leq e^{\frac{c}{1 - \alpha}}$. By the Cauchy--Schwarz inequality, we have
\begin{align*}
\E \biggl| \sum_{\substack{n \leq x \\ e^{\frac{c}{1-\alpha}} < P(n) \leq y}} f(n) \biggr| &= \E \E_{\alpha} \biggl| \sum_{\substack{1 < m \leq x \\ p | m \Rightarrow p \in (e^{\frac{c}{1-\alpha}}, y]}} f(m) \sum_{\substack{n \leq x/m \\ P(n) \leq e^{\frac{c}{1 - \alpha}}}} f(n) \biggr| \\
& \leq \E \biggl( \E_{\alpha} \biggl| \sum_{\substack{1 < m \leq x \\ p | m \Rightarrow p \in (e^{\frac{c}{1-\alpha}}, y]}} f(m) \sum_{\substack{n \leq x/m \\ P(n) \leq e^{\frac{c}{1 - \alpha}}}} f(n) \biggr|^2 \biggr)^{1/2} \\ 
& = \E \biggl( \sum_{\substack{1 < m \leq x \\ p | m \Rightarrow p \in (e^{\frac{c}{1-\alpha}}, y]}} \biggl| \sum_{\substack{n \leq x/m \\ P(n) \leq e^{\frac{c}{1 - \alpha}}}} f(n) \biggr|^2 \biggr)^{1/2} .
\end{align*}
We wish to replace the sum over $m$ by an integral, which will in turn allow us to apply Parseval's identity. First, we introduce a ``dummy integral'' over $t$, writing
\begin{align*}
\E \biggl| \sum_{\substack{n \leq x \\ e^{\frac{c}{1-\alpha}} < P(n) \leq y}} f(n) \biggr| & \leq \E \biggl( \sum_{\substack{1 < m \leq x \\ p | m \Rightarrow p \in (e^{\frac{c}{1-\alpha}}, y]}} \biggl| \sum_{\substack{n \leq x/m \\ P(n) \leq e^{\frac{c}{1 - \alpha}}}} f(n) \biggr|^2 \biggr)^{1/2} \\
& = \E \biggl( \sum_{\substack{1 < m \leq x \\ p | m \Rightarrow p \in (e^{\frac{c}{1-\alpha}}, y]}} \frac{X}{m} \int_{m}^{m(1+\frac{1}{X})} \biggl| \sum_{\substack{n \leq x/m \\ P(n) \leq e^{\frac{c}{1 - \alpha}}}} f(n) \biggr|^2 \, dt \biggr)^{1/2} .
\end{align*}
The first term on the right-hand is
\begin{multline*}
\ll \E \biggl( \sum_{\substack{1 < m \leq x \\ p | m \Rightarrow p \in (e^{\frac{c}{1-\alpha}}, y]}} \frac{X}{m} \int_{m}^{m(1+\frac{1}{X})} \biggl| \sum_{\substack{n \leq x/t \\ P(n) \leq e^{\frac{c}{1 - \alpha}}}} f(n) \biggr|^2 \, dt \biggr)^{1/2} + \\ 
\E \biggl( \sum_{\substack{1 < m \leq x \\ p | m \Rightarrow p \in (e^{\frac{c}{1-\alpha}}, y]}} \frac{X}{m} \int_{m}^{m(1+\frac{1}{X})} \biggl| \sum_{\substack{x/t < n \leq x/m \\ P(n) \leq e^{\frac{c}{1 - \alpha}}}} f(n) \biggr|^2 \, dt \biggr)^{1/2} .
\end{multline*}
We show that the second term is small. By the Cauchy--Schwarz inequality, we can bound it from above by
\[
\biggl( \sum_{\substack{1 < m \leq x \\ p | m \Rightarrow p \in (e^{\frac{c}{1-\alpha}}, y]}} \frac{X}{m} \int_{m}^{m(1+\frac{1}{X})} \sum_{\substack{x/t < n \leq x/m \\ P(n) \leq e^{\frac{c}{1 - \alpha}}}} 1 \, dt \biggr)^{1/2} \ll \biggl( \sum_{\substack{1 < m \leq x \\ p | m \Rightarrow p \in (e^{\frac{c}{1-\alpha}}, y]}} \sum_{\substack{x/m(1 + \frac{1}{X}) < n \leq x/m \\ P(n) \leq e^{\frac{c}{1 - \alpha}}}} 1\biggr)^{1/2}  .
\]
We make the choice $X = \log \log x$. Change the order of summation, and apply Lemma~\ref{l: smooths below x/d vs below x} and Lemma~\ref{l: smooths in short intervals estimate} to get that (Noting carefully that the conditions for Lemma~\ref{l: smooths in short intervals estimate} are satisfied in our application below)
\begin{align*}
& \ll \biggl( \sum_{\substack{n \leq x \\ P(n) \leq e^{\frac{c}{1 - \alpha}}}} \sum_{\substack{x/n(1 + \frac{1}{X}) < m \leq x/n \\ p \mid m \Rightarrow p \in (e^{\frac{c}{1 - \alpha}}, y]}} 1  \biggr)^{1/2} \ll \biggl( (1 - \alpha) \sum_{\substack{n \leq x \\ P(n) \leq e^{\frac{c}{1 - \alpha}}}} \frac{1}{X} \Psi (x/n, y) \biggr)^{1/2} \\
& \ll \biggl( \frac{\Psi (x, y) (1-\alpha)}{X} \sum_{\substack{n \leq x \\ P(n) \leq e^{\frac{c}{1 - \alpha}}}} \frac{1}{n^{\alpha}} \biggr)^{1/2} \ll \frac{\Psi(x,y)^{1/2}}{X^{1/2}} = \sqrt{\frac{\Psi(x,y)}{\log \log x}} .
\end{align*}
At this stage, we have
\begin{align*}
\E \biggl| \sum_{\substack{n \leq x \\ e^{\frac{c}{1-\alpha}} < P(n) \leq y}} f(n) \biggr| &\ll \E \biggl( \sum_{\substack{1 < m \leq x \\ p | m \Rightarrow p \in (e^{\frac{c}{1-\alpha}}, y]}} \frac{X}{m} \int_{m}^{m(1+\frac{1}{X})} \biggl| \sum_{\substack{n \leq x/t \\ P(n) \leq e^{\frac{c}{1 - \alpha}}}} f(n) \biggr|^2 \, dt \biggr)^{1/2} + \sqrt{\frac{\Psi(x,y)}{\log \log x}} \\
& = \E \biggl( \int_{e^{\frac{c}{1-\alpha}}}^{x} \biggl| \sum_{\substack{n \leq x/t \\ P(n) \leq e^{\frac{c}{1 - \alpha}}}} f(n) \biggr|^2 \sum_{\substack{t/(1 + \frac{1}{X})\leq m \leq t \\ p | m \Rightarrow p \in (e^{\frac{c}{1-\alpha}}, y]}} \frac{X}{m}  \, dt \biggr)^{1/2} + \sqrt{\frac{\Psi(x,y)}{\log \log x}}  .
\end{align*}
Applying Lemma~\ref{l: smooths in short intervals estimate}, we have
\[
\sum_{\substack{t/(1 + \frac{1}{X})\leq m \leq t \\ p | m \Rightarrow p \in (e^{\frac{c}{1-\alpha}}, y]}} \frac{X}{m} \ll ( 1- \alpha ) \frac{\Psi (t,y)}{t} ,
\]
uniformly for $t \geq y$. In the remaining range $e^{\frac{c}{1 - \alpha}} \leq t < y$, the estimate follows using a simple sieve argument (see, for example,~\cite[Theorem~3.6]{MV2007}), sieving out all primes $p \leq (e^{\frac{c}{1 - \alpha}})^{1/10}$, say. Making the change of variables $t = x/z$, we have
\begin{multline*}
\E \biggl| \sum_{\substack{n \leq x \\ e^{\frac{c}{1-\alpha}} < P(n) \leq y}} f(n) \biggr| \ll
(1 - \alpha)^{1/2} \E \biggl( \int_{1}^{x/e^{\frac{1}{1 - \alpha}}} \biggl| \sum_{\substack{n \leq z \\ P(n) \leq e^{\frac{c}{1 - \alpha}}}} f(n) \biggr|^2  \Psi (x/z, y) \, \frac{dz}{z} \biggr)^{1/2} \\ + \sqrt{\frac{\Psi(x,y)}{\log \log x}} .
\end{multline*}
Applying Lemma~\ref{l: smooths below x/d vs below x}, we have $\Psi(x/z, y) \ll \Psi(x, y)/z^{\alpha}$, and completing the integral, we obtain
\begin{equation}\label{equ:large y pre planch}
\begin{split}
\E \biggl| \sum_{\substack{n \leq x \\ e^{\frac{c}{1-\alpha}} < P(n) \leq y}} f(n) \biggr| \ll \bigl( \Psi (x,y) (1 - \alpha ) \bigr)^{1/2} \E \biggl( \int_{1}^{\infty} \biggl| \sum_{\substack{n \leq z \\ P(n) \leq e^{\frac{c}{1 - \alpha}}}} f(n) \biggr|^2 \frac{dz}{z^{1 + \alpha}} \biggr)^{1/2} \\ + \sqrt{\frac{\Psi(x,y)}{\log \log x}} .
\end{split}
\end{equation}
The second term is negligible for our upper bound. We next manipulate the first term so that it can be bounded by using Lemma~\ref{l:Euler product on half line}, and for this purpose, we let $z = e^{\frac{c}{1 - \alpha}}$, so that $F_{z} (s) = \prod_{p \leq e^{\frac{c}{1 - \alpha}}} \bigl( 1 - \frac{f(p)}{p^{s}} \bigr)^{-1}$. Applying Parseval's identity, Lemma~\ref{l:ha result}, followed by Hölder's inequality, we obtain
\begin{align*}
\E \biggl( \int_{1}^{\infty} \biggl| \sum_{\substack{n \leq z \\ P(n) \leq e^{\frac{c}{1 - \alpha}}}} f(n) \biggr|^2 \frac{dz}{z^{1 + \alpha}} \biggr)^{1/2} &= \E \biggl( \int_{-\infty}^{\infty} \biggl| \frac{F_{z} (\frac{\alpha}{2} + it)}{\frac{\alpha}{2} + it} \biggr|^2 \, dt \biggr)^{1/2} \\ 
&\leq \biggl( \E \biggl( \int_{-\infty}^{\infty} \biggl| \frac{F_{z} (\frac{\alpha}{2} + it)}{\frac{\alpha}{2} + it} \biggr|^2 \, dt \biggr)^{2/3} \biggr)^{3/4}.
\end{align*}
The above expectation satisfies 
\begin{align*}
\E \biggl( \int_{-\infty}^{\infty} \biggl| \frac{F_{z} (\frac{\alpha}{2} + it)}{\frac{\alpha}{2} + it} \biggr|^2 \, dt \biggr)^{2/3} & \ll \E \biggl( \sum_{n \in \Z} \frac{1}{|n|^{2} + 1} \int_{n - 1/2}^{n + 1/2} \Bigl| F_{z} \Bigl( \frac{\alpha}{2} + it \Bigr) \Bigr|^2 \, dt \biggr)^{2/3}  \\
& \ll \sum_{n \in \Z} \frac{1}{n^{4/3} + 1} \E \biggl( \int_{n - 1/2}^{n + 1/2} \Bigl| F_{z} \Bigl( \frac{\alpha}{2} + it \Bigr) \Bigr|^2 \, dt \biggr)^{2/3} \\
& \ll \E \biggl( \int_{-1/2}^{1/2} \Bigl| F_{z} \Bigl( \frac{\alpha}{2} + it \Bigr) \Bigr|^2 \, dt \biggr)^{2/3}.
\end{align*}
where to obtain the last line, we have used translation invariance in law. Finally, by Lemma~\ref{l:Euler product on half line} with $q = 2/3$, we have
\[
\E \biggl( \int_{-1/2}^{1/2} \Bigl| F_{z} \Bigl( \frac{\alpha}{2} + it \Bigr) \Bigr|^2 \, dt \biggr)^{2/3} \ll \Biggl( \frac{1}{(1 - \alpha) \sqrt{\log \frac{1}{1-\alpha}}} \Biggr)^{2/3},
\]
so overall this gives
\[
\E \biggl( \int_{1}^{\infty} \biggl| \sum_{\substack{n \leq z \\ P(n) \leq e^{\frac{c}{1 - \alpha}}}} f(n) \biggr|^2 \frac{dz}{z^{1 + \alpha}} \biggr)^{1/2} \ll \frac{1}{(1 - \alpha)^{1/2} \bigl( \log \frac{1}{1 - \alpha} \bigr)^{1/4}} .
\]
Inserting this into~\eqref{equ:large y pre planch}, we deduce that
\[
\E \biggl| \sum_{\substack{n \leq x \\ e^{\frac{c}{1-\alpha}} < P(n) \leq y}} f(n) \biggr| \ll \frac{\Psi (x,y)^{1/2}}{\Bigl(\log \frac{1}{1-\alpha(x,y)} \Bigr)^{1/4}},
\]
which completes the proof of Proposition~\ref{p:small u upper bound}.
\end{proof}

\section{Sums over very smooth numbers: Proof of Theorem~\ref{t: small y}}\label{Sec: small y}
In this section, we give a proof of Theorem~\ref{t: small y}. As mentioned in the introduction, to prove parts $(i)$ and $(ii)$, we again apply Perron's formula similarly to in Section~\ref{sec: moderate}, but this time we get a saving for our Euler product integrals from Lemma~\ref{l:exp for small y}. To give the reader some clarity, the first range, $D \leq y \leq \log x / \log \log x$, is chosen to allow for immediate application of this lemma to the whole Euler product $\E |F_y (\alpha(x,y) / 2 + it)|$.

\begin{proof}[Proof of Theorem~\ref{t: small y}]

We first prove $(i)$ and $(ii)$, with $(iii)$ following from a straightforward exponential sum argument. Identically to as is done at the beginning of the proof of Theorem~\ref{Thm: main quantitative} (see the start of Section~\ref{sec: moderate}), we begin by applying Perron's formula~\cite[Corollary~5.3]{MV2007} to obtain
\[
\sum_{\substack{n \leq x \\ P(n) \leq y}} f(n) = \frac{1}{2 \pi i} \int_{1 - i x}^{1 + ix} F_y (s) x^s \frac{ds}{s} + O (\log x),
\]
which actually holds uniformly for $2 \leq y \leq x$ (in fact, our initial manipulations will all hold on this range). As in the previous section, we define $\alpha \coloneqq \alpha(x,y)$, and we now shift the main term to the line $\Re (s) = \alpha/2$. Similarly to in the previous section, we then have
\begin{multline*}
\E \Bigl| \sum_{\substack{n \leq x \\ P(n) \leq y}} f(n) \Bigr| \ll  x^{\alpha/2} \int_{-x}^{x} \E \biggl| \frac{F_y (\alpha / 2 + it)}{\alpha / 2 + it} \biggr| \, dt + \log x \\ 
+ \frac{1}{x} \int_{\alpha/2}^1 \bigl( \E|F_y (\sigma + ix)| + \E|F_y (\sigma - ix)| \bigr) x^{\sigma} \, d \sigma  .
\end{multline*}
Now applying Cauchy--Schwarz followed by the trivial bound~\eqref{equ: trivial EP bound} for the Euler products in the last term, we find that the last integral term above is 
\[
\ll \frac{1}{x} \int_{\alpha/2}^{1} \zeta(2 \sigma, y)^{1/2} x^{\sigma} \, d \sigma = \frac{1}{x}  \int_{\alpha/2}^{1} \bigl( x^{2\sigma} \zeta (2 \sigma, y) \bigr)^{1/2} \, d \sigma \ll 1,
\]
uniformly for $2 \leq y \leq x$. Here we have used the fact that $x^{2\sigma} \zeta (2\sigma, y)$ is maximised on our range of integration when $\sigma = 1$, which can be shown by differentiating, as in the derivation on the saddle point. From this, we obtain the bound
\begin{equation}\label{equ: before sep y ranges}
\E \Bigl| \sum_{\substack{n \leq x \\ P(n) \leq y}} f(n) \Bigr| \ll  x^{\alpha/2} \int_{-x}^{x} \E \biggl| \frac{F_y (\alpha / 2 + it)}{\alpha / 2 + it} \biggr| \, dt + \log x .
\end{equation}
uniformly for $2 \leq y \leq x$. We now separate into the ranges mentioned in the theorem, beginning with part $(i)$, where we assume that $D \leq y \leq \log x / \log \log x$ for some sufficiently large constant $D \geq 5$. By Lemma~\ref{l:saddle point approx uniform}, in this range, we certainly have $\alpha (x,y) \asymp \frac{y}{\log x \log y}$ and in particular this is $o(\frac{1}{\log y})$. Therefore (assuming $x$ is sufficiently large), we can apply Lemma~\ref{l:exp for small y} to obtain the bound
\[
\E \Bigl| \sum_{\substack{n \leq x \\ P(n) \leq y}} f(n) \Bigr| \ll  x^{\alpha/2} \prod_{p \leq y} \biggl( \frac{1}{3} \log \biggl( \frac{1}{1-p^{-\alpha/2}} \biggr) \biggr) \int_{-x}^{x} \frac{dt}{|\alpha / 2 + it|} + \log x . 
\]
Again using the fact that $\alpha (x,y) \asymp \frac{y}{\log x \log y}$, we see immediately that the integral in the first term is $\ll \log x$, so we obtain the bound
\[
\E \Bigl| \sum_{\substack{n \leq x \\ P(n) \leq y}} f(n) \Bigr| \ll  x^{\alpha/2} (\log x) \prod_{p \leq y} \biggl( \frac{1}{3} \log \biggl( \frac{1}{1-p^{-\alpha/2}} \biggr) \biggr) + \log x . 
\]
This is the first result of part $(i)$ of the Theorem. We now wish to compare this to the trivial bound, $\sqrt{\Psi(x,y)}$. Applying the bound~\eqref{equ:rankin to true count comparison}, we find that 
\[
\E \Bigl| \sum_{\substack{n \leq x \\ P(n) \leq y}} f(n) \Bigr| \ll  \sqrt{\Psi(x,y)} (\log x)^{5/4} (\log y)^{1/4} \prod_{p \leq y} \biggl( \frac{1}{3} (1 - p^{-\alpha})^{1/2} \log \biggl( \frac{1}{1-p^{-\alpha/2}} \biggr) \biggr) + \log x . 
\]
We now apply the bound $\sqrt{1 - x} \log \bigl( \frac{1}{1-\sqrt{x}} \bigr) \leq 3$ for $x \in (0,1)$ to remove the contribution from primes $p \leq \sqrt{y}$. An alternative method would have been to apply the Cauchy--Schwarz bound~\eqref{equ: trivial EP bound} to this part of the Euler product in~\eqref{equ: before sep y ranges}. The last display is therefore
\[
\ll  \sqrt{\Psi(x,y)} (\log x)^{5/4} (\log y)^{1/4} \prod_{\sqrt{y} < p \leq y} \biggl( \frac{1}{3} (1 - p^{-\alpha})^{1/2} \log \biggl( \frac{1}{1-p^{-\alpha/2}} \biggr) \biggr) + \log x . 
\]
To complete the proof of part $(i)$, we give fairly straightforward bounds for the terms in the product. First of all, by Taylor expansion and Lemma~\ref{l:saddle point approx uniform}, if $D$ is sufficiently large, we have 
\[
(1-p^{-\alpha})^{1/2} \leq \frac{3}{2} \sqrt{\frac{y \log p}{(\log x)(\log y)}} \leq \frac{3}{2} \sqrt{\frac{y}{\log x}},
\]
which holds uniformly on this range of $y$ when $x$ is sufficiently large. Here we have used the fact that $\log p \leq \log y$. Furthermore, we have
\[
\log \biggl( \frac{1}{1-p^{-\alpha/2}} \biggr) \leq \log \biggl( \frac{3 (\log x) (\log y)}{y \log p} \biggr) \leq 2 \log \biggl( \frac{ \log x}{y} \biggr),
\]
which holds similarly when $x$ is sufficiently large. Here we have used the fact that $\log p \geq (\log y) / 2$. Inserting these results into the product, we have
\[
\E \Bigl| \sum_{\substack{n \leq x \\ P(n) \leq y}} f(n) \Bigr| \ll  \sqrt{\Psi(x,y)} (\log x)^{5/4} (\log y)^{1/4} \biggl( \sqrt{\frac{y}{\log x}} \log \biggl( \frac{\log x}{y} \biggr) \biggr)^{\pi(y) - \pi(\sqrt{y})} + \log x . 
\]
It can be verified by straightforward estimates that this always gives better than square-root cancellation on the entire range $D \leq y \leq \log x / \log \log x$, so long as $D$ is sufficiently large. This completes the proof of part $(i)$ of Theorem~\ref{t: small y}.

We now move on to proving part $(ii)$ of Theorem~\ref{t: small y}, first fixing the large constant $C>0$ appearing there, and assuming that $\log x / \log \log x \leq y \leq C \log x$. We then return to equation~\eqref{equ: before sep y ranges}. Now, for the largest primes in this range, we cannot directly apply Lemma~\ref{l:exp for small y}, seeing as the saddle point is slightly too large relative to $y$. To make the lemma applicable, we handle the large primes trivially, and we will only apply Lemma~\ref{l:exp for small y} to the small primes. Therefore, we break up the Euler product in~\eqref{equ: before sep y ranges}, writing
\[
\E \Bigl| \sum_{\substack{n \leq x \\ P(n) \leq y}} f(n) \Bigr| \ll  x^{\alpha/2} \int_{-x}^{x} \E |F_{y^{\delta}} (\alpha / 2 + it) | \E \prod_{y^{\delta} < p \leq y} \biggl| 1 - \frac{f(p)}{p^{\alpha/2 + it}} \biggr| \frac{1}{|\alpha / 2 + it|} \, dt + \log x .
\]
which holds for any $0 < \delta < 1$. Note that in our range of $y$, Lemma~\ref{l:saddle point approx uniform} implies $\alpha (x,y) \ll_C \frac{1}{\log y}$. Therefore, taking the constant $\delta$ to be sufficiently small in terms of $C$, we can apply Lemma~\ref{l:exp for small y} to the first Euler product over small primes. To handle the contribution from the large primes, we apply the Cauchy--Schwarz inequality (identically to~\eqref{equ: trivial EP bound}) which tells us that 
\[
\E \prod_{y^{\delta} < p \leq y}  \biggl| 1 - \frac{f(p)}{p^{\alpha/2 + it}} \biggr| \leq \prod_{y^{\delta} < p \leq y} \biggl( 1 - \frac{1}{p^{\alpha}} \biggr)^{-1/2} .
\]
Inserting these results, we obtain
\[
\E \Bigl| \sum_{\substack{n \leq x \\ P(n) \leq y}} f(n) \Bigr| \ll  x^{\alpha/2} \prod_{p \leq y^{\delta}} \frac{1}{3} \log \biggl( \frac{1}{1 - p^{- \alpha / 2}} \biggr) \prod_{y^{\delta} < p \leq y} \biggl( 1 - \frac{1}{p^{\alpha}} \biggr)^{-1/2}  \int_{-x}^{x}\frac{1}{|\alpha / 2 + it|} \, dt + \log x .
\]
Similarly to before, the integral is $\ll \log x$. We now compare this to $\sqrt{\Psi(x,y)}$, again using~\eqref{equ:rankin to true count comparison}. Seeing as the contributions from the large primes cancel, we obtain
\[
\E \Bigl| \sum_{\substack{n \leq x \\ P(n) \leq y}} f(n) \Bigr| \ll  \sqrt{\Psi(x,y)} (\log x)^{5/4} (\log y)^{1/4} \prod_{p \leq y^{\delta}} \biggl( \frac{1}{3} (1 - p^{-\alpha})^{1/2} \log \biggl( \frac{1}{1 - p^{- \alpha / 2}} \biggr) \biggr)  + \log x .
\]
Finally, we use that fact that $\sqrt{1 - x} \log \bigl( \frac{1}{1-\sqrt{x}} \bigr) \leq 3/2$ for $x \in (0,1)$. Applying this inequality with $x = p^{-\alpha}$, we obtain
\[
\E \Bigl| \sum_{\substack{n \leq x \\ P(n) \leq y}} f(n) \Bigr| \ll  \sqrt{\Psi(x,y)} (\log x)^{5/4} (\log y)^{1/4} 2^{- \pi (y^{\delta})}  + \log x .
\]
Seeing as $\log x / \log \log x \leq y \leq C \log x$, we deduce the bound
\[
\E \Bigl| \sum_{\substack{n \leq x \\ P(n) \leq y}} f(n) \Bigr| \ll \sqrt{\Psi (x,y)/\exp ((\log x)^{c})}, 
\]
where $c$ is some small constant depending on $C$ (since $\delta$ depended only on $C>0$). This is precisely $(ii)$ of Theorem~\ref{t: small y}.

We finish by proving the final part, $(iii)$, of Theorem~\ref{t: small y}, which follows by a wasteful, yet simple, exponential sum estimate, where we only obtain a saving from a single prime, $f(2)$. Suppose that $y$ is bounded, and let $p_r$ be the $r$-th prime which is the largest prime $\le y$. Then
\[
\sum_{\substack{n \leq x \\ P(n) \leq y}} f(n) = \sum_{\substack{n\le x\\ P(n)\le p_r}} f(n)  =  \sum_{\substack{m\le x \\ p|m \Rightarrow 3\le p \le p_r}} f(m) \sum_{\substack{\ell\le x/m \\p|\ell \Rightarrow p=2 }} f(\ell). 
\]
Letting $\tilde{\E}$ be the expectation conditioned on $f(2)$, and applying the Cauchy-Schwarz inequality, we have
\begin{align*}
\E\Big|\sum_{\substack{n\le x\\ P(n)\le p_r}} f(n)\Big| &\le \E \biggl( \tilde{\E} \Bigl| \sum_{\substack{m\le x \\ p|m \Rightarrow 3\le p \le p_r}} f(m) \sum_{\substack{\ell\le x/m \\p|\ell \Rightarrow p=2 }} f(\ell) \Bigr|^2 \biggr)^{1/2}  \\ 
&= \E \bigg(\sum_{\substack{m\le x \\ p|m \Rightarrow 3\le p \le p_r}} \Bigl| \sum_{\substack{\ell\le x/m \\p|\ell \Rightarrow p=2 }} f(\ell) \Bigr|^{2}\bigg)^{1/2}.
\end{align*}
The inner sum can be written explicitly by using the law of the Steinhaus random variable, hence
\[ 
\E\Big|\sum_{\substack{n\le x\\ P(n)\le p_r}} f(n)\Big| \leq \int_{-1/2}^{1/2}\bigg(\sum_{\substack{m\le x \\ p|m \Rightarrow 3\le p \le p_r}} \Bigl| \sum_{\substack{k\le \log_2 (x/m) }} e(k\theta) \Bigr|^{2}\bigg)^{1/2} d\theta.
\]
For $|\theta| \leq 1/\log x$, we use a trivial bound on the inner sum and find that the integral is 
\[\ll \frac{1}{\log x } \cdot ((\log x)^{r-1} \cdot (\log x)^{2})^{1/2} = (\log x)^{\frac{r-1}{2}},\]
which is acceptable since $\Psi(x, p_r) \asymp_r (\log x)^{r}$ (see, for example,~\cite[Theorem~1]{Ennola}). For $\theta$ in the complementary range $1/\log x \leq |\theta| \leq 1/2$, we can first bound the geometric series over $k$ by $\ll \frac{1}{|\theta|}$ which leads to a bound on the integral
\[\ll (\log x)^{r-1} \int_{\frac{1}{\log x} \leq |\theta| \leq 1/2 } \frac{d\theta}{\theta}  \ll (\log x)^{\frac{r-1}{2}} \log \log x.  \]
Combining these estimates and comparing to the aforementioned bound $\Psi(x, p_r) \asymp_r (\log x)^{r}$, this completes the proof of the theorem.  
\end{proof}

\appendix
\section{Proof of Lemma~\ref{l: smooths in short intervals estimate}}\label{s:appendix}

\begin{proof}[Proof of Lemma~\ref{l: smooths in short intervals estimate}]
As mentioned, this lemma is a natural extension of~\cite[Theorem~3]{Hildebrand}. To begin, we take $y \geq 2$, $1 \leq u \leq \exp \bigl( (\log y)^{3/5 - \varepsilon} \bigr)$, $y^{-1/3} \leq h \leq 1/2$ and let $\frac{1}{\log y} \leq \delta \leq \frac{100}{\log u}$. Note carefully that in our proof,  mimicking Hildebrand~\cite{Hildebrand}, we redefine $h$ to be the relative length of the sum (i.e. here $h$ plays the role of $h/x$ in the statement of Lemma~\ref{l: smooths in short intervals estimate}). These conditions therefore correspond to the case where $e^{(\log \log x)^{5/3 + \varepsilon}} \leq y \leq x$, $\frac{1}{\log y} \leq \delta \leq \frac{100}{\log u}$ and $x/y^{1/3} \leq h \leq x/2$ in the lemma.

We make use of the standard notation \[\Psi_q (x,y) = \# \{ n \leq x : \, p \mid n \Rightarrow (p,q) = 1 \text{ and } p \leq y \},\] 
and in our proof we fix $q = \prod_{p \leq y^\delta} p$. For any $z \geq y$, we have
\[
\int_1^z \frac{\Psi_q (z,y) - \Psi_q (t,y)}{t} \, dt = \sum_{\substack{n \leq z \\ p|n \Rightarrow y^{\delta} < p \leq y}} \log n ,
\]
and upon writing $\log n = \sum_{d \mid n} \Lambda(d)$, it follows that
\[
\Psi_q (z,y) \log z = \sum_{\substack{p^m \leq z \\ y^\delta < p \leq y}} \Psi_q \bigl( z/p^m , y ) \log p + \int_1^z \frac{\Psi_q (t,y)}{t} \, dt .
\]
Applying the formula above twice and taking the difference gives
\begin{multline*}
\Psi_q (y^u, y) - \Psi_q (y^u (1-h), y) = \frac{1}{u \log y} \biggl( \int_{y^u (1-h)}^{y^u} \frac{\Psi_q (t,y)}{t} \, dt - \Psi_q (y^u (1-h),y) \log \biggl( \frac{1}{1-h} \biggr) \\
+ \sum_{\substack{p^m \leq y^u \\ y^{\delta} < p \leq y}} \Bigl( \Psi_q \bigl( y^u/p^m , y \bigr) - \Psi_q \bigl( y^u (1-h)/p^m, y \bigr) \Bigr) \log p \biggr) .
\end{multline*}
Rewriting the term $\Psi_q (y^u (1-h),y) \log ( \frac{1}{1-h} )$ into the integral leads to
\begin{multline}\label{equ:smooths short int}
\Psi_q (y^u, y) - \Psi_q (y^u (1-h), y) = \frac{1}{u \log y} \biggl( \int_{y^u (1-h)}^{y^u} \frac{\Psi_q (t,y) - \Psi_q (y^u (1-h),y)}{t} \, dt \\
+ \sum_{\substack{p^m \leq y^u \\ y^{\delta} < p \leq y}} \Bigl( \Psi_q \bigl( y^u/p^m , y \bigr) - \Psi_q \bigl( y^u (1-h)/p^m, y \bigr) \Bigr) \log p \biggr) .
\end{multline}
We define $\Delta_h (\delta, y, u) \geq 0$ to satisfy
\[
\Psi_q (y^u, y) - \Psi_q (y^u (1-h), y) = \Delta_h (\delta, y, u)  \frac{h y^u \rho(u)}{\delta \log y} ,    
\]
and we set
\[
\Delta_h^* (\delta, y, u) \coloneqq \sup_{1/2 \leq u' \leq u} \Delta_h (\delta, y, u') .
\]
We will show that uniformly for $ e^{(\log \log x)^{5/3 + \varepsilon}} \leq y \leq x, y^{-1/3} \leq h \leq 1/2, \frac{1}{\log y} \leq \delta \leq \frac{100}{\log u}$: 
\begin{equation}\label{equ: final}
    \Delta_h^* (\delta, y, u) \ll 1.
\end{equation}
If this is proved, then one can see that by using $\Psi(x, y)\gg x\rho(u)$ in our range of parameters (\cite{Hensley1985}), then the Lemma is proved. 

Starting from now, we focus on proving \eqref{equ: final}.
We first consider the integral term in~\eqref{equ:smooths short int}. We have
\begin{align*}
\int_{y^u (1-h)}^{y^u} \frac{\Psi_q (t,y) - \Psi_q (y^u (1-h),y)}{t} \, dt & \leq \frac{h y^u \rho(u )\Delta_h (\delta, y, u)}{\delta \log y} \log \biggl( \frac{1}{1-h} \biggr) \\
& \ll \frac{h y^u \rho(u) \Delta_h^* (\delta, y, u)}{\delta \log y} ,
\end{align*}
and the implied constant is absolute. We turn our attention to
\begin{multline*}
\sum_{\substack{y < p^m \leq y^u \\ y^{\delta} < p \leq y}} \Bigl( \Psi_q \bigl( y^u/p^m , y \bigr) - \Psi_q \bigl( y^u (1-h)/p^m, y \bigr) \Bigr) \log p \\
= \frac{1}{\delta \log y} \sum_{\substack{y < p^m \leq y^u \\ y^{\delta} < p \leq y}} \Delta_h \biggl( \delta, y, u - \frac{\log p^m}{\log y} \biggr) \frac{h y^u \log p}{p^m} \rho \biggl( u - \frac{\log p^m}{\log y} \biggr) .
\end{multline*}
We apply~\cite[Lemma~3]{Hildebrand}, which states that
\[
\sum_{\substack{y < p^m \leq y^u \\ p \leq y}} \frac{\log p}{p^m} \rho \biggl( u - \frac{\log p^m}{\log u} \biggr) \ll \rho (u) ,
\]
uniformly for $y \geq 2$ and $u \leq y^{1/4}$, so we deduce that
\[
\sum_{\substack{y < p^m \leq y^u \\ y^{\delta} < p \leq y}} \Bigl( \Psi_q \bigl( y^u/p^m , y \bigr) - \Psi_q \bigl( y^u (1-h)/p^m, y \bigr) \Bigr) \log p  \ll \frac{h y^u \rho (u)  \Delta_h^* (\delta, y, u) }{\delta \log y} .
\]
Inserting our estimates into~\eqref{equ:smooths short int} and rewriting the left hand side, we deduce that $\Delta_h (\delta,y,u) =$
\[
\frac{1}{\rho(u) u \log y} \sum_{\substack{p^m \leq y \\ y^{\delta} < p \leq y}} \Delta_h \biggl( \delta, y, u - \frac{\log p^m}{\log y} \biggr) \rho \biggl( u - \frac{\log p^m}{\log y} \biggr) \frac{\log p}{p^m} + O \biggl( \frac{\Delta_h^* (\delta, y, u)}{u \log y} \biggr) ,
\]
uniformly for $y \geq 2$ and $u \leq y^{1/4}$. We deduce that $\Delta_h (\delta,y,u)$ is smaller than or equal to
\[
\frac{1}{u \rho(u) \log y} \sum_{\substack{p^m \leq y}} \Delta_h \biggl( \delta, y, u - \frac{\log p^m}{\log y} \biggr) \rho \biggl( u - \frac{\log p^m}{\log y} \biggr) \frac{\log p}{p^m} +  O \biggl( \frac{\Delta_h^* (\delta, y, u)}{u \log y} \biggr).
\]
Applying~\cite[Lemma~4]{Hildebrand}, which states that for every fixed $\varepsilon>0$ and uniformly for $y \geq 2, u \geq 1$ and $0 \leq \theta \leq 1$ the following holds
\begin{multline*}
\sum_{p^m \leq y^{\theta}} \rho\left(u-\frac{\log p^m}{\log y}\right) \frac{\log p}{p^m} =(\log y) \int_{u-\theta}^u \rho(t) \, d t \\
+ O_{\varepsilon}\Bigl( \rho(u) \bigl( 1+u \log ^2(u+1) \exp (-(\log y)^{3 / 5-\varepsilon} ) \bigr) \Bigr)  ,
\end{multline*}
to find that, for $u \leq \exp \bigl( (\log y)^{3/5 - \varepsilon} \bigr)$, we have
\begin{multline*}
\Delta_h (\delta,y,u) \leq \frac{1}{u \rho(u)} \biggl[ \Delta_h^* \bigl( \delta, y, u \bigr)  \int_{u - 1/2}^{u} \rho (t) \, dt 
+ \Delta_h^* \bigl( \delta, y, u - 1/2 \bigr) \int_{u-1}^{u-1/2} \rho (t) \, dt \biggr] \\ +  O \biggl( \frac{\Delta_h^* (\delta, y, u)}{u \log y} \biggr) . 
\end{multline*}
Seeing as $\int_{u-1}^u \rho (t) \, dt = u \rho(u)$ (see~\cite[Lemma~1]{Hildebrand}), we have
\begin{align*}
1 & = \frac{1}{u \rho (u)} \int_{u-1/2}^{u} \rho(t) \, dt + \frac{1}{u \rho (u)} \int_{u-1}^{u-1/2} \rho(t) \, dt \\
& \eqqcolon \alpha (u) + \bigl( 1 - \alpha (u) \bigr).
\end{align*}
Hence
\[
\Delta_h (\delta,y,u) \leq \biggl( \Delta_h^* \bigl( \delta, y, u \bigr) \alpha (u) +  \Delta_h^* \bigl( \delta, y, u - 1/2 \bigr) \bigl( 1 - \alpha (u) \bigr) \biggr) + O \biggl( \frac{\Delta_h^* (\delta, y, u)}{u \log y} \biggr) . 
\]
As is noted in the proof of Theorem 1 in~\cite{Hildebrand}, we have $\alpha (u) \leq 1/2$ (this follows from~\cite[Lemma~1]{Hildebrand}). We deduce that the quantity
\begin{multline*}
\frac{1}{2}\left(\Delta^*(\delta, y, u)+\Delta^*\left(\delta, y, u-\frac{1}{2}\right)\right)-\left(\alpha(u) \Delta^*(\delta, y, u)+(1-\alpha(u)) \Delta^*\left(\delta, y, u-\frac{1}{2}\right)\right) \\
=\left(\frac{1}{2}-\alpha(u)\right)\left(\Delta^*(\delta, y, u)-\Delta^*\left(\delta, y, u-\frac{1}{2}\right)\right)
\end{multline*}
is nonnegative. Therefore,
\[
\Delta_h (\delta,y,u) \leq \frac{1}{2} \Bigl( \Delta_h^* \bigl( \delta, y, u \bigr) + \Delta_h^* \bigl( \delta, y, u - 1/2 \bigr) \Bigr) +  O \biggl( \frac{\Delta_h^* (\delta, y, u)}{u \log y} \biggr).
\]
By taking the sup of $\Delta_h (\delta,y,u')$ over $1/2 u' u$ and using the above bound, we find that there exists some constant $C > 0$,
\[
\Delta_h^* \bigl( \delta, y, u \bigr) \leq \biggl( 1 + C\frac{1}{u \log y} \biggr) \Delta_h^* \bigl( \delta, y, u - 1/2 \bigr) .
\]
Iterating this inequality, we deduce that
\begin{equation}\label{equ: iterated Delta bound}
\Delta_h^* \bigl( \delta, y, u \bigr) \leq e^{O(\frac{\log u}{\log y})} \Delta_h^* \bigl( \delta, y, 1 \bigr) .
\end{equation}
Recall that for $c \in [1/2, 1]$, 
\[
\Delta_h (\delta, y, c) = \frac{\delta \log y}{h y^c \rho(c)} \Bigl( \Psi_q (y^c , y) - \Psi_q (y^c (1-h), y) \Bigr) .
\]
Finally, seeing as
\[
\Psi_q (y^c, y) - \Psi_q (y^c (1-h), y) = \sum_{\substack{y^c (1-h) < n \leq y^c \\ p \mid n \Rightarrow p \in (y^\delta, y]}} 1 \leq \sum_{\substack{y^c (1-h) < n \leq y^c \\ p \mid n \Rightarrow p > y^{\delta}}} 1 , 
\]
and it follows from a simple sieve argument (for example, by~\cite[Theorem~3.6]{MV2007}) that
\[
\sum_{\substack{y^c (1-h) < n \leq y^c \\ p \mid n \Rightarrow p > y^{\delta}}} 1 \ll \frac{h y^c}{\delta \log y},
\]
uniformly for $c \in [1/2,1]$, $h \geq y^{-1/3}$, and $\frac{1}{\log y} \leq \delta \leq \frac{1}{10}$. We therefore have
\[
\Delta_h (\delta, y, c) \ll 1,
\]
uniformly for $c \in [1/2,1]$. Combining this with
~\eqref{equ: iterated Delta bound}, we deduce that 
\[\Delta_h^* \bigl( \delta, y, u \bigr) \leq e^{O(\frac{\log u}{\log y})} \ll  1,\]
which follows from the fact that $y \geq e^{(\log \log x)^{5/3 + \varepsilon}}$. This completes the proof of \eqref{equ: final} and the lemma is proved. 

\end{proof}

\subsubsection*{Rights Retention}
For the purpose of open access, the authors have applied a Creative Commons Attribution (CC-BY) licence to any Author Accepted Manuscript version arising from this submission.

\bibliographystyle{plain}
	\bibliography{main}{}
\end{document}